\documentclass[microtype]{gtpart}     
\agtart
\usepackage{graphicx}
\usepackage{subcaption}
\usepackage{tikz}
\usepackage{tikz-cd}
\usetikzlibrary{matrix,arrows,decorations.pathmorphing}
\usepackage{thmtools}

\title{An Analogue of Milnor's Invariants for Knots in 3-Manifolds}

\author{Miriam Kuzbary}
\givenname{Miriam}
\surname{Kuzbary}
\address{School of Mathematics\\
Georgia Institute of Technology\\
686 Cherry Street\\
Atlanta, GA 30332-0160 USA}
\email{kuzbary@gatech.edu}
\urladdr{mkuzbary3.math.gatech.edu} 


\arxivreference{math.GT/arXiv:1910.12108}

\newtheorem{thm}{Theorem}[section]
\newtheorem{prop}[thm]{Proposition}
\newtheorem{lem}[thm]{Lemma}
\newtheorem{cor}[thm]{Corollary}

\theoremstyle{definition}
\newtheorem{definition}[thm]{Definition}

\theoremstyle{remark}
\newtheorem{remark}[thm]{Remark}

\numberwithin{equation}{section}

\newcommand{\Ml}{\mathop\#\limits^{\ell}S^1 \times S^2}

\newcommand{\inlineMl}{\#^{\ell}S^1 \times S^2}

\begin{document}

\begin{abstract}    

Milnor's invariants are some of the more fundamental oriented link concordance invariants; they behave as higher order linking numbers and can be computed using combinatorial group theory (due to Milnor), Massey products (due to Turaev and Porter), and higher order intersections (due to Cochran). In this paper, we generalize the first non-vanishing Milnor's invariants to oriented knots inside a closed, oriented $3$-manifold $M$. We call this the Dwyer number of a knot and show methods to compute it for null-homologous knots inside connected sums of $S^1 \times S^2$.  We further show in this case the Dwyer number provides the weight of the first non-vanishing Massey product in the knot complement in the ambient manifold.  Additionally, we prove the Dwyer number detects a family of null-homotopic knots K in $\inlineMl$ bounding smoothly embedded disks in $\natural^{\ell} D^2 \times S^2$ which are not concordant to the unknot.

\end{abstract}

\maketitle


\section{Introduction}

In order to use knots and links to understand $4$-manifold topology, we examine their equivalence classes up to a $4$-dimensional relation called concordance.  We say two $n$-component links $L_0$ and $L_1$ are concordant if there are $n$ smooth, disjoint, properly embedded annuli in $S^3 \times I$ with boundary $L_0 \times \{0\} \sqcup -L_1 \times \{1\}$ and an $n$-component link $L \subset S^3$ is slice if it bounds $n$ smooth, properly embedded disks in $B^4$.  Note that a knot is a link with one component. Knots modulo concordance with the operation connected sum forms the knot concordance group $\mathcal{C}$ introduced by Fox and Milnor in \cite{foxmilnor} and results about knot concordance have had many useful implications for the study of $4$-manifolds. For example, the result of Dehn surgery on two concordant knots, K and J, is two $3$-manifolds which are homology cobordant by a $4$-manifold obtained by replacing $S^1 \times D^2 \times I$ in $S^3 \times I$ by $D^2 \times S^1 \times I$ with appropriate framing.

From this perspective, it is a pressing issue that every $3$-manifold is the result of Dehn surgery on a link, and not necessarily a knot. Moreover, there are concordance invariants of links, such as linking number, which are not simply a generalization of some quantity for knots but are a fundamentally distinct idea unique to these objects with multiple components. To this end, this paper addresses the following questions:

\begin{itemize}
\item \textit{Can concordance invariants of links in $S^3$ give information about knots in more complicated $3$-manifolds obtained by Dehn surgery on some (possibly different) link?}
\item \textit{Can these invariants give more subtle information about links which are distinct in concordance than previously known techniques?}
\end{itemize}

We address these questions using Milnor's link concordance invariants defined in \cite{milnor54}, which are useful precisely because they detect subtle linking data not obtained from simply generalizing invariants for knots. In particular, linking number is a Milnor invariant of weight 2. The higher weight Milnor's invariants can even detect nontriviality of links where each component is the unknot and removal of any single component results in the unlink, for example, the Milnor invariants of weight 3 (also known as the triple linking number) obstruct the Borromean rings from being slice.  Note that Milnor's invariants are topological concordance invariants and not invariants of smooth concordance; however, in this article we work in the smooth category.

Given a knot $K$ in a $3$-manifold M, we define a new integer-valued concordance invariant called the Dwyer number, denoted $D(K, \gamma)$, where $\gamma$ is a curve homologous to $K$ with certain technical properties.  This invariant relies on certain subgroups of the second homology of the knot complement which encode information about the lower central series quotients of the fundamental group of the knot complement. As detailed in Section \ref{section:background}, this is exactly the information the Milnor's invariants give us for links.  We say two knots $K_0$ and $K_1$ in $M$ are concordant if there is a smooth, properly embedded annulus in $M \times I$ with boundary $K_0 \times \{0\} \sqcup -K_1 \times \{1\}$ where $-K_1$ is $K_1$ with the opposite orientation. 

\begin{restatable*}{thm}{concinvt}\label{thm:concinvt} $D(K, \gamma)$ is an invariant of knot concordance in $M \times I$.
\end{restatable*}

The construction of $D(K, \gamma)$ is quite different from that of Milnor as the combinatorial group theory tools by Magnus in \cite{magnus} exploited by Milnor in his original work do not directly apply.   Recall the lower central series of a group $G$ is defined recursively by $G_1 = G$ and $G_{m}=[G, G_{m-1}]$.  Milnor's original definition used nice properties of the lower central series quotients of free groups; in our definition we constructed an invariant using the lower central series quotients of the knot group using a somewhat different approach. However, will see in Section \ref{section:dwyernum} that in certain cases this new invariant shares many of the useful properties of Milnor's invariants, in particular, its relationship with fundamental group of the link complement in the ambient $3$-manifold. Note that Milnor's invariants have been generalized to knots in specific non-trivial homotopy classes of Seifert fiber spaces by D. Miller \cite{miller} and knots in prime manifolds which are non-trivial in homotopy by Heck \cite{heck}.

In the case where $K$ is null-homologous in $M$, we can take $\gamma$ to be an unknot and denote this invariant $D(K)$. The Dwyer number detects which elements of $H_2(M)$ can be represented by $G_m$-surfaces where $G=\pi_1(M \setminus \nu(K), *)$ as detailed Section \ref{section:dwyernumgen}. These objects can be viewed as the continuous image of $2$-complexes historically called gropes. For reasons clear from the name,  in this work we will instead refer to a grope as a surface tower as suggested by Scott Carter and Ian Agol. When $ K \subset \#^{\ell}S^1 \times S^2$, $D(K)$ can be viewed as a generalization of the lowest order non-vanishing Milnor's $\bar{\mu}$-invariant of a link as indicated by Theorem \ref{thm:longitude}.  Throughout this paper, we will refer to the group $\pi_1(M \setminus \nu(L), *)$ of a knot or link $L$ in a $3$-manifold $M$ by $G(L, M)$,  and suppress $M$ from the notation when the ambient manifold is clear from context.

\begin{restatable*}{thm}{longitudethm} \label{thm:longitude}
If $K \subset \inlineMl$ is a null-homologous knot with $D(K) = q$, then the longitude of $K$ lies in $G(K)_{q-1}$ and not $G(K)_{q}$.
\end{restatable*}

In this special case $K \subset \inlineMl$, $D(K)$ also gives the weight of the first non-vanishing Massey product in the knot complement just as Milnor's invariants do for links in $S^3$.

\begin{restatable*}{prop}{masseythm} \label{thm:massey}
If $D(K) = q$ then in $H^*(\inlineMl \setminus \nu(K))$ all non-vanishing Massey products are weight $ \geq q$. 
\end{restatable*}

The original motivation for this project comes from a construction in Heegaard Floer homology called the knotification of a link. This construction takes an $n$-component link $L \subset S^3$ and forms a knot $\kappa(L)$ inside $\#^{n-1}S^1 \times S^2$. The knotification of a link is particularly important as it appears in the original definition by Ozsv\'{a}th and Szab\'{o} in \cite{ozszhololinkfloer} of the knot Floer complex for a link $L \subset S^3$, and the resulting complex is exactly the link Floer complex of $L$ with all $U_i$ variables set equal to a single $U$ variable and the grading appropriately shifted. Our work arose from examining the knotification of a link $L \subset S^3$ and determining what linking data from $L$ remained detectable in the knotification in a connected sum of copies of $S^1 \times S^2$. In future work, we hope to use these ideas to detect this higher order linking data in the Heegaard Floer homology setting itself. To this end, we can use our work to predict properties of some knotified links based on the Milnor invariants of a link in $S^3$.

\begin{restatable*}{prop}{knotifiedbound}Let $L \subset S^3$ be an $n$-component link whose first nonzero Milnor invariant $\overline{\mu}_L(I)$ is weight $q$ for some positive integer $q$. Then $D(\kappa(L)) \geq \left \lceil{\frac{q-1}{n}}\right \rceil $.
\end{restatable*}

While it is interesting for those of us who care about Milnor's invariants to have this generalization, it is compelling to exhibit situations where this invariant can distinguish knots not previously able to be distinguished.

Recall for a knot $K$ in $S^3$, $K$ being slice in $B^4$ is equivalent to $K$ being concordant to the unknot in $S^3 \times I$. This is not true for knots in a general $3$-manifold $M$ bounding a $4$-manifold as demonstrated in the following theorem. Results in this work about the concordance of knots in $(\inlineMl) \times I$ will indicate how different the study of knot concordance in this setting is from that of knots in $S^3$. Furthermore, these results motivate the definition of a new link concordance group in upcoming joint work with Matthew Hedden.

\begin{restatable*}{prop}{slicenotconc}\label{thm:slicenotconc}Let $L \subset S^3$ be an $n+1$-component link with slice components such that an $n$-component sublink $U$ is the unlink. Then, the image of $L \setminus U$ after performing $0$-surgery on $U$ is a knot $K\subset \#^n S^1 \times S^2$ which is slice in $\natural^n D^2 \times S^2$ and, if $L$ had at least one non-vanishing Milnor invariant of some weight $I$, is not concordant to the unknot in $ \big(\#^n S^1 \times S^2 \big) \times I$.
\end{restatable*}

\begin{figure}[h] 
\centering
  \includegraphics[width=.7\linewidth]{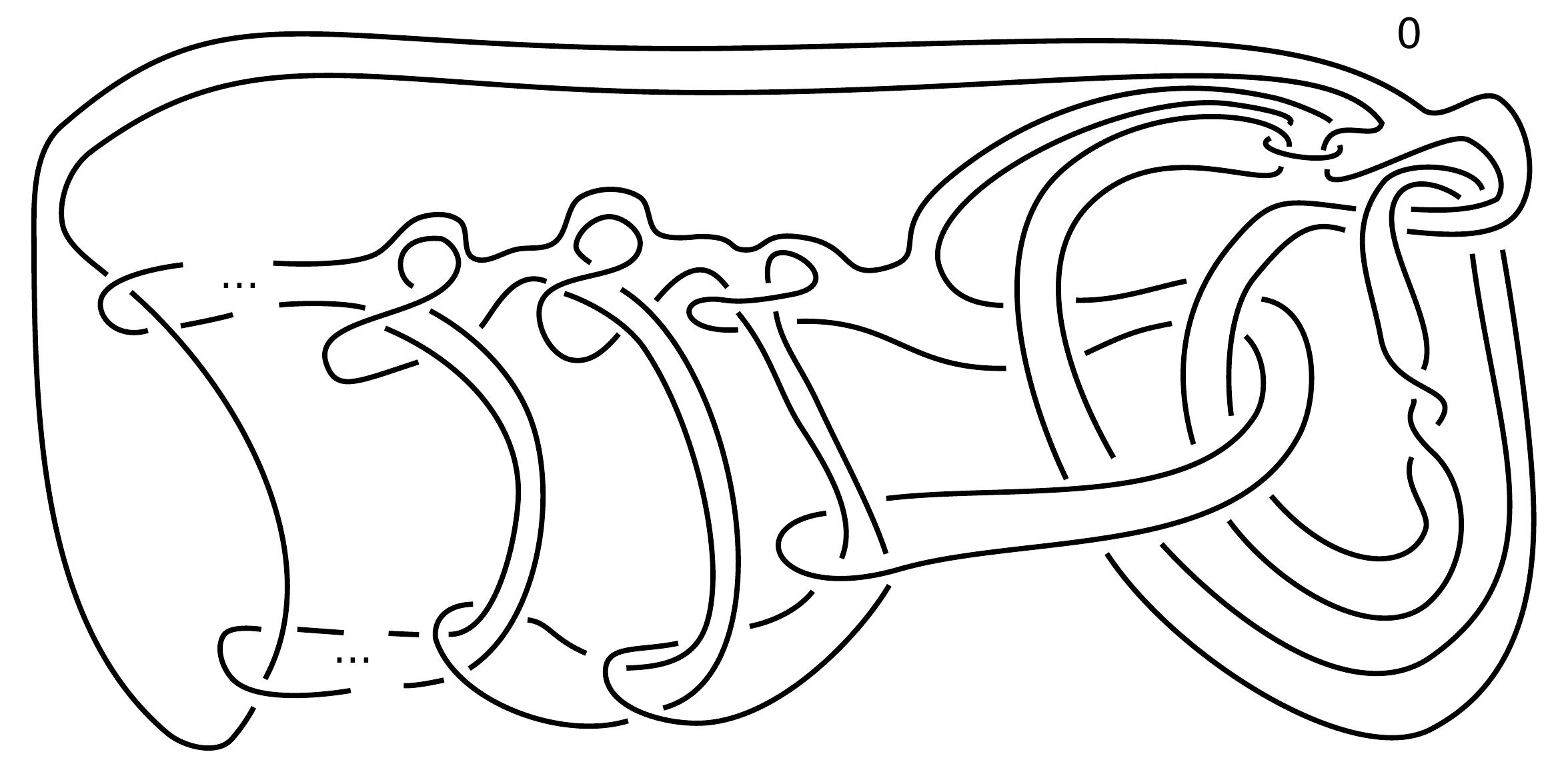}
\caption{A family of knots $L_{2n}$ in $S^1 \times S^2$ with $D(L_{2n})=2n-1$, where $n$ refers to the number of copies of the repeating tangle in the middle of the knot.}
\label{fig:L2ninmfd}
\end{figure}

\begin{restatable*}{thm}{realization} \label{thm:realization}
For each $i\in \Z_{+}$ there is a null-homotopic family of distinct knots $$\{L_{2n} \subset \#^i S^1 \times S^2 \mid n \in \Z_{+}\}$$ with trivial algebraic self-linking number (in the sense of Schneiderman in \cite{schn}) which bound a smooth, properly embedded disk in $\natural^i D^2 \times S^2$ but is not concordant to the trivial knot in $(\#^i S^1 \times S^2) \times I$.
\end{restatable*}

During the course of this work, other researchers in low dimensional topology worked on a similar problem and also exhibited examples of knots in $S^1 \times S^2$ which are slice in $D^2 \times S^2$ and not concordant to the unknot in $(S^1 \times S^2) \times I$. The covering link techniques of Davis, Nagel, Park, and Ray in \cite{DNPR} and Friedl, Nagel, Orson, and Powell in \cite{FNOP} distinguish $L_{2n}$ from being concordant to the trivial link, but do not distinguish $L_{2n}$ from $L_{2k}$ where $n \neq k$ as the linking numbers associated to their covering links are all $-3$. Additionally, the algebraic self-linking numbers of \cite{schn} are all trivial from a straightforward calculation using the obvious null-homotopy. 

Finally,  as detailed by Celoria in \cite{celoriaac} the knot concordance group acts on the set of concordance classes of knots in a $3$-manifold $M$ by local knotting.  The quotient by this action is the set of so called almost-concordance classes of knots,  and Celoria conjectured that the trivial homotopy class in each $3$-manifold contains infinitely many knot representatives,  distinct up to almost-concordance.  This question was addressed by \cite{FNOP}, \cite{DNPR}, and Yildiz in \cite{yildiz} at around the same time the bulk of the work in this paper was completed,  and the conjecture was proven in \cite{FNOP}.  As demonstrated by the examples in Theorem \ref{thm:realization}, the techniques from these papers are not sufficient to distinguish our examples from each other, and it is in this case that the higher order linking data detected by the Dwyer number is useful.  Therefore,  we present the following proposition to indicate that the Dwyer number can be used to study almost-concordance classes in more detail. 

\begin{restatable*}{prop}{almostconc}\label{thm:almostconc}
Fix $\ell \in \Z_{\geq 0}$ and let $K$ be a null-homologous knot in $\inlineMl$.  Then $D(K)$ is an invariant of almost concordance in $\inlineMl$. 
\end{restatable*}

\subsection{Future Questions}
\begin{enumerate}
\item Is there a relationship between $D(K,\gamma)$ for $K \subset M$ and the Milnor number of a 3-manifold M defined by Cochran and Melvin? 
\item Are there examples of knots $K \subset M$ with finite $D(K, \gamma)$ where $\pi_1(M, \star)$ is not free?
\item For any closed, oriented 3-manifold M, and each homology class $x \in H_1(M)$ is there a smooth embedded curve $\gamma \in x$ such that,  for any knot $K \in x$, there is a homomorphism $\psi$ from $G(\gamma)$ to $G(K)$ inducing an isomorphism from $H_1(M \setminus \nu(\gamma))$ to $H_1(M \setminus \nu(K))$?
\item Is there an unordered version of the first non-vanishing Milnor's invariants defined by taking the quotient of the free group by the action of the symmetric group and looking at the image of longitudes in this result? If such a thing is well-defined, can you use this idea to help define higher order tree valued invariants in the sense of Schneiderman and Teichner in \cite{schnteich}?
\item By Theorem \ref{thm:massey},  $D(K)$ for $K \subset \#^{\ell} S^1 \times S^2$ gives a lower bound on the weight of non-vanishing Massey products on $\#^{\ell} S^1 \times S^2 \setminus \nu(K)$.  Which of these products is well-defined and are they also concordance invariants?
\end{enumerate}

\subsection{Acknowledgments} The author was supported by an NSF Graduate Research Fellowship under Grant No. 1450681 as well as an AAUW American Fellowships Dissertation Fellowship. Additionally, the author was partially supported by NSF grant DMS-1745583. The author would like to thank her advisor, Shelly Harvey, for her mentorship during this project and Matthew Hedden for his support and many wonderful math conversations. The author would also like to express her gratitude to Tim Cochran for his guidance on Milnor's invariants early in her graduate career. The author would further like to thank Laura Starkston, Jennifer Hom, and her anonymous reviewers for taking the time to give valuable feedback on this work.

\subsection{Outline} In Section \ref{section:background}, we outline the necessary background on Milnor's invariants. In Section \ref{section:dwyernumgen}, we define our invariant for the case of knots in oriented, closed $3$-manifolds and prove it is an invariant of concordance in $M \times I$. Finally, in Section \ref{section:dwyernum} we prove this invariant generalizes important properties of Milnor's invariants and express a bound for the invariant in certain cases in terms of the Milnor's invariants of an associated link.

\section{Background} \label{section:background}

To give context for the Dwyer number and justification for why it generalizes the useful properties of Milnor's invariants, the following is a short overview of the original definition of these invariants and why they are useful. While the original definition of Milnor's invariants with non-repeating indices by Milnor in \cite{milnor54} and in general in \cite{milnor57} is somewhat laborious, as we will see in this section the construction naturally extends the idea of linking number in the group theoretic setting in order detect subtle higher order linking data of a link. 
Milnor's invariants arise in three contexts: 1) combinatorial group theory, 2) cohomology, and 3) intersection theory. The connection between the group theoretic and cohomological perspectives using Massey products was conjectured by Milnor in his original definitions, but was not proven until much later by Turaev in \cite{turaev} and independently by Porter in \cite{porter}. Later, using work of Stein in \cite{stein}, Cochran showed one can exploit the duality between Massey products and iterated intersections of surfaces in order to compute the first nonvanishing Milnor's invariants in \cite{cochranmemoir}. In this section, we survey definitions and major theorems about these invariants.

Note that the fundamental group of a link complement is not itself a concordance invariant: there are many slice links with non-trivial fundamental group. However, we can extract concordance data from the fundamental group of a link through a clever use of the lower central series. Recall that for an arbitrary link in the 3-sphere, we can repackage the linking numbers between its components as a statement about the homology of the complement of the link. In this article, we will refer to a 0-framed longitude of a knot or link simply as the longitude of the knot or link. 

\begin{remark}\label{homologyandlinking}
If $L$ is an $n$-component oriented link in $S^3$ with $L_i$ the $0$-framed longitude of the $i^{th}$ component of $L$, then
$$[L_i] = \Sigma_{j=1}^n \text{lk}(L_i,L_j) \cdot x_j \in \text{H}_1(S^3 \setminus \nu(L)) = G(L) / [G(L),G(L)]$$
where $x_j$ represents the $j^{th}$ meridian of $L$. 
\end{remark}

From this perspective, one might ask if there is concordance information detected by other quotients of the fundamental group. Theorem \ref{thm:casson} indicates there is. 

\begin{thm}[Casson \cite{casson}] \label{thm:casson}
If $L_1$ and $L_2$ are concordant links in $S^3$ then $G(L_1)/G(L_1)_m$ and $G(L_2)/G(L_2)_m$ are isomorphic for all $m \geq 1$.
\end{thm}

As a result, we have an entire family of nilpotent groups which can tell us information about link concordance. Furthermore, we have the following important result concerning maps on group homology and these nilpotent quotients.

\begin{thm}[Stallings' Integral Theorem \cite{stallings}] \label{thm:stallings}
Let $\varphi: A \rightarrow B$ be a homomorphism inducing an isomorphism on $\text{H}_1( - ; \mathbb{Z} )$ and an epimorphism on $\text{H}_2( - ; \mathbb{Z} )$. Then, for each n, $\varphi$ induces an isomorphism $A / A_m \cong B / B_m$.
\end{thm}

In order to use these theorems to detect whether a link is non-trivial, we must understand how the structure of the group quotients relate to the topology we are using them to study. The following lemma is well known, but we include it here along with a proof to motivate later definitions.

\begin{lem}\label{lem:foundation}
Let $L$ be an $n$-component link in $S^3$.
\begin{enumerate}
\item If $L$ has only one component (i.e. $L$ is a knot), then $G(L)/G(L)_m \cong \Z$ for all $m$.
\item If $L$ is the $n$-component unlink, then $G(L)$ is a free group generated by the meridians of $L$.
\item If $L$ is an $n$-component slice link and $F$ is a free group on $n$ letters, then $G(L)/G(L)_m \cong F / F_m $ for all $m$.
\end{enumerate}
\end{lem}

\begin{proof}
\begin{enumerate}
\item Consider the homomorphism $f: \Z \rightarrow \pi_1(S^3 \setminus L, *)$ defined by sending the generator $1$ of $\Z$ to a meridian of $L$. Since $L$ is a knot, this map induces isomorphisms on homology. Thus, by Theorem \ref{thm:stallings}, $f$ induces isomorphisms on the lower central series quotients.
\item This is a straightforward calculation using the Seifert-van Kampen theorem.
\item Since $L$ is a slice link, it is concordant to the $n$-component unlink. By Theorem \ref{thm:casson} and the previous statement, $G/G_m$ is isomorphic to $F/F_m$ for all $m$.
\end{enumerate}

\end{proof}

This lemma first tells us that the nilpotent quotients of the fundamental group of a knot in $S^3$ will not tell us any new concordance information, as we might expect as there is nothing for the knot to ``link" with in the first place since $S^3$ is simply connected. Therefore, in the case of knots and links in $S^3$, these lower central series quotients will only be useful for studying links. As we will show, this is not the case for knots in arbitrary $3$-manifolds. Furthermore, Lemma \ref{lem:foundation} indicates that if we want to determine if an $n$-component link is non-trivial in concordance, we should examine whether the lower central series quotients of its link group are isomorphic to the lower central series quotients of a corresponding free group. One major difficulty is the lower central series quotients of free groups on $n > 1$ letters become unwieldy quite quickly. We can detect whether these quotients are isomorphic using a clever group presentation. Recall that the fundamental group consists of based homotopy classes, however, since in this article we work in nilpotent quotients of fundamental groups and changing the basepoint corresponds to conjugation, we will suppress the basing from discussion as in Section 3 of \cite{cochranmemoir}.

As observed by Milnor, the meridian homomorphism $\varphi: F \rightarrow G/G_m$ where $F=\langle x_1, ..., x_n \rangle$ is a free group on $n$ letters and $\varphi$ sends a generator $x_i$ to a meridian $\mu_i$ is surjective and for any longitude $l_i$ there is a word $R_m(l_i) \in F$ such that $\varphi(R_m(l_i)) \equiv l_i \mod G_m$. In the language of Cochran, we call any such element $R_m(l_i)$ an $m$-rewrite of $l_i$.

\begin{thm}[Milnor \cite{milnor57}] \label{thm:presentation} Let $L \subset S^3$ be an $n$-component link. For $m \geq 1$, there is a presentation
$$\frac{G(L)}{G(L)_m} \cong \langle x_1,..., x_n \mid [x_i, R_m(l_i)], 1 \leq i \leq n, F_m \rangle $$
where $x_i$ represents an $i^{th}$ meridian of L, $l_i$ represents the $i^{th}$ longitude of L, $F$ is the free group on the letters $x_1, ..., x_n$, and $R_m(l_i)$ is an $m$-rewrite of $l_i$.

\end{thm}

 In summary, it is clear that for any $m \in \Z$ with $m \geq 2$, we have $$ \frac{G(L)}{G(L)_m} \cong \frac{F}{F_m} \iff [x_i, R_m(l_i)] \in F_m \text{ for each }i \iff R_m(l_i) \in F_{m-1}.$$

The remaining piece of the definition of Milnor's invariants relies on work of Magnus; there is an embedding of the free group on $n$ letters into the power series ring $\mathbb{Z}[[X_1, ...X_n]]$ of $n$ non-commuting variables defined on generators by

\begin{align*}
M(x_i) &= 1+X_i \\
M(x^{-1}_i) &= 1 - X_i+X_i^2-X_i^3+...
\end{align*} A word in the free group is in $F_q$ if and only if all the coefficients of its Magnus expansion of degree $q$ or less are 0, this proof is detailed by Magnus, Karrass, and Solitar in Chapter 5 of \cite{mks}.

\begin{definition}[Milnor \cite{milnor57}] Let $L \subset S^3$ be an oriented, ordered link. The Milnor invariants of $L$ are integers $\overline\mu_L(i_1, ..., i_k)$ each corresponding to a multi-index $(i_1, ..., i_m)$ where $i_j \in \{1, ..., n\}$ and are defined as follows.

Let $l_{i_m}$ be the $i_m^{th}$ longitude of $L$ and let $R_m(l_{i_m})$ be an $m$-rewrite as defined above. By Theorem \ref{thm:presentation}, this group element can be represented by a word $w$ in meridians $x_1, ..., x_n$. The Magnus expansion of this word is $$M(w) = 1 + \Sigma_I \epsilon_I X^{I}$$ where the sum is taken over all possible multi-indices $I = (j_1, ..., j_m)$ and $X^I$ is shorthand for $X^{j_1}...X^{j_m}$. We will refer to the coefficient of $X^I$ in $M(w)$ as $\epsilon_I(w)$.  Then,
$$\overline\mu(i_1, ..., i_m) := \epsilon_{i_1, ..., i_{m-1}}(R_m(l_{i_m})).$$

We will refer to the length of the multi-index $|i_1, ..., i_m|$ as the weight (or order) of the invariant $\overline\mu(i_1, ..., i_m)$. The integer  is well-defined if all the Milnor's invariants of weight less than $m$ are 0, otherwise, this integer is defined to be the residue class modulo
$$\Delta = \text{gcd}\{\overline\mu(I) \mid I \in \tilde{I}\}$$
where $\tilde{I}$ is obtained from $i_1, ..., i_m$ by removing one index and cyclically permuting the other indices.
\end{definition}

\begin{thm}[Casson \cite{casson}]
The $\overline{\mu}$-invariants are concordance invariants for oriented, ordered links in $S^3$.
\end{thm}

Notice that Milnor's invariants are only defined modulo certain Milnor's invariants of smaller weight. Therefore, in this work we are primarily concerned with the first Milnor's invariants which are non-zero as in \cite{cochranmemoir}.

\begin{thm}[Turaev \cite{turaev}, Porter \cite{porter}]
Let $L \subset S^3$ be an oriented, ordered, $n$-component link and let $L_i$ refer to its $i^{th}$ component. Let $u_i \in H^1(S^3 \setminus L_i)$ be the Alexander dual of the generator of $H_1(L_i)$ determined by the orientation of $L_i$. For $i$, $j \in \{ 1, ..., n\}$, set $\gamma_{i,j}$ equal to the Lefschetz dual of the element in $H_1(S^3, L_i \cup L_j)$ determined by a path from $L_i$ to $L_j$.

Let $(l_1, ..., l_p)$ be a sequence of integers with $l_k \in \{1, ..., n\}$. If all Milnor invariants of $L$ of weight less than $p$ are $0$, then the Massey product $\langle u_{l_1}, ..., u_{l_p} \rangle$ in $S^3 \setminus L$ is defined and
$$\langle u_{l_1}, ..., u_{l_p} \rangle = (-1)^p \overline{\mu}_L(l_1, ..., l_p) \gamma_{l_1, l_p}.$$
\end{thm}

Notice this means we can recover the appropriate Milnor invariants by evaluating $\langle u_{l_1}, ..., u_{l_p} \rangle$ on the corresponding boundary components of $S^3 \setminus \nu(L)$.  Recall that Massey products are generalized cup products; just as cup products are dual to intersections, Massey products are dual to higher order intersections. Cochran's work  allows us to compute Milnor's invariants using surface systems which precisely encode the correct higher order intersections to calculate the lowest weight non-vanishing Milnor invariants of a link $L$ \cite{cochranmemoir}.

\section{Concordance data in the lower central series} \label{section:dwyernumgen}

In Section \ref{section:background} we introduced Milnor's invariants which detect subtle higher order linking data for a link in the $3$-sphere using the lower central series quotients of the fundamental group of the link complement. We further indicated why Milnor's invariants for a knot in $S^3$ are not useful. It is natural to ask whether this approach could detect concordance information for knots and links in $3$-manifolds other than $S^3$. As this work demonstrates, the nilpotent quotients of the fundamental group of the complement of a knot in a more complicated $3$-manifold still proves to be quite useful. In Section \ref{section:background}, we introduced a theorem of Casson which generalizes to the following well-known result which we will prove later in the section. 

\begin{prop}\label{prop:hcob} If $K_1$ and $K_2$ are concordant knots inside a closed, oriented $3$-manifold $M$, then $G(K_1)/G(K_1)_m \cong G(K_2)/G(K_2)_m$ for all $m$.
\end{prop}

To define the Dwyer number, we first go back to the original proof of Theorem \ref{thm:stallings} and examine why the condition relating maps on group homology to nilpotent quotients of the respective groups is sufficient to induce isomorphisms on nilpotent quotients. The proof of this theorem relies on the following sequence.

\begin{thm}[Stallings \cite{stallings}]\label{thm:stallingsseq}
For a group $G$ with normal subgroup $N$ there is a natural exact sequence

$$H_2(G) \rightarrow H_2\bigg(\frac{G}{N}\bigg) \rightarrow \frac{N}{[G, N]} \rightarrow H_1(G) \rightarrow H_1\bigg(\frac{G}{N}\bigg) \rightarrow 0$$.

In the case that $N = G_m$, we have 

$$H_2(G) \rightarrow H_2\bigg(\frac{G}{G_m}\bigg) \rightarrow \frac{G_m}{G_{m+1}} \rightarrow H_1(G) \rightarrow H_1\bigg(\frac{G}{G_m}\bigg) \rightarrow 0.$$

\end{thm}

 This sequence indicates the significance of the kernel of the map $H_2(G) \rightarrow H_2(G / G_m )$. This kernel is called the $m^{th}$ Dwyer subgroup of $G$, and we denote it by $\Phi_m(G)$.

\begin{definition}\label{def:cc} Let $N$ be a normal subgroup of a group $G$. As defined by Cochran and Harvey in \cite{cochranharvey}, define $\Phi^N(G)$ as the image of the map $H_2(N) \rightarrow H_2(G)$. Equivalently, $\Phi^N(G)$ is the subgroup of classes which can be represented by maps of closed, oriented surfaces $f:\Sigma \rightarrow \Phi^N(G) $ such that $f_*(\pi_1(\Sigma, *) \subset N$. We call such surfaces $N$-surfaces. We can similarly define $N$-surfaces for a topological space $X$ with fundamental group $G$ and see that $\Phi^N(G)=\Phi^N(K(G,1))$. In this language, by Freedman and Teichner in Lemma 2.3 of \cite{freedmanteichner} and Definition 1.5 in \cite{cochranharvey} we see that $\Phi_m(G)$ is exactly $\Phi^{G_m}(K(G,1))$. We will therefore refer to elements of $\Phi_m(X)$ as $G_m$-surfaces.

As proven in Lemma 2.3 and further detailed in pages 533-537 in \cite{freedmanteichner}, a $G_m$-surface has a nice interpretation as the topological reformulation of an element in the $m^{th}$ term of the lower central series of $G$. More precisely, the elements of  $\Phi_m(X)$ are maps of $2$-complexes which we call surface towers \footnote{As mentioned in the introduction, historically such a $2$-complex has been referred to as a grope.} defined recursively as follows. 

\begin{definition}[See \cite{freedmanteichner}]
A surface tower has a class $m=1, ..., \infty$. We will first define a surface tower with boundary, and a closed surface tower will be the result of replacing a $2$-cell in $S^2$ with an $m$-surface tower. A surface tower (with boundary) is a pair $(\text{2-complex}, \text{circle})$.  A class $1$ surface tower is the pair  $(\text{circle}, \text{circle})$, and a class $2$ surface tower is a compact, oriented surface $\Sigma$ with one boundary component. For finite $m$, an $m$-surface tower is defined inductively as follows. Let $\{a_i, b_i \mid i =1, ..., g(\Sigma)\}$ be a standard symplectic basis of simple closed curves for $\Sigma$. For any $p_i, q_i \in \Z_+$ with $p_i+q_i \geq m$ and $p_{i_0}+q_{i_0}=m$ for at least one index $i_0$, an $m$-surface tower is constructed by gluing a $p_i$-surface tower to each $a_i$ and $q_i$-surface tower to each $b_i$. An $\infty$-surface tower is a nested union of $m$-towers for all $m>1$. For a surface tower of class at least two, we will refer to each layer of surfaces as a stage, with the first stage being the initial surface whose curves subsequent layers are glued to. 
\end{definition} 

Recall that the $m^{th}$ term of the lower central series of $G$ is generated by simple iterated commutators $[x_1, [x_2, [x_1,..., x_m]]]$ as detailed in Section 5.3 of \cite{mks}. When $G$ is the fundamental group of a manifold $M$ we can realize a commutator in $G$ geometrically as the boundary of a surface of some genus continuously mapped into $M$ (though not necessarily embedded). Therefore, if we have an iterated commutator, we can realize it geometrically as the continuous image of a $2$-complex as above where new layers of surfaces are added on carefully. This idea motivates the following definition.

\begin{definition}A half-surface tower of order $m$ is inductively defined to be obtained from a surface $\Sigma$ by attaching half-surface towers of order $(m-1)$ to a half-symplectic basis $\{a_i\}$ of $\Sigma$. 
\end{definition}

As noted in Lemma 2.3 of \cite{freedmanteichner}, any map of an $m$-surface tower is also represented by a map of an order $m$ half-surface tower, therefore, we will mainly concern ourselves with these representatives.

\end{definition}
\begin{figure}[ht]
\centering
\includegraphics[scale=.2]{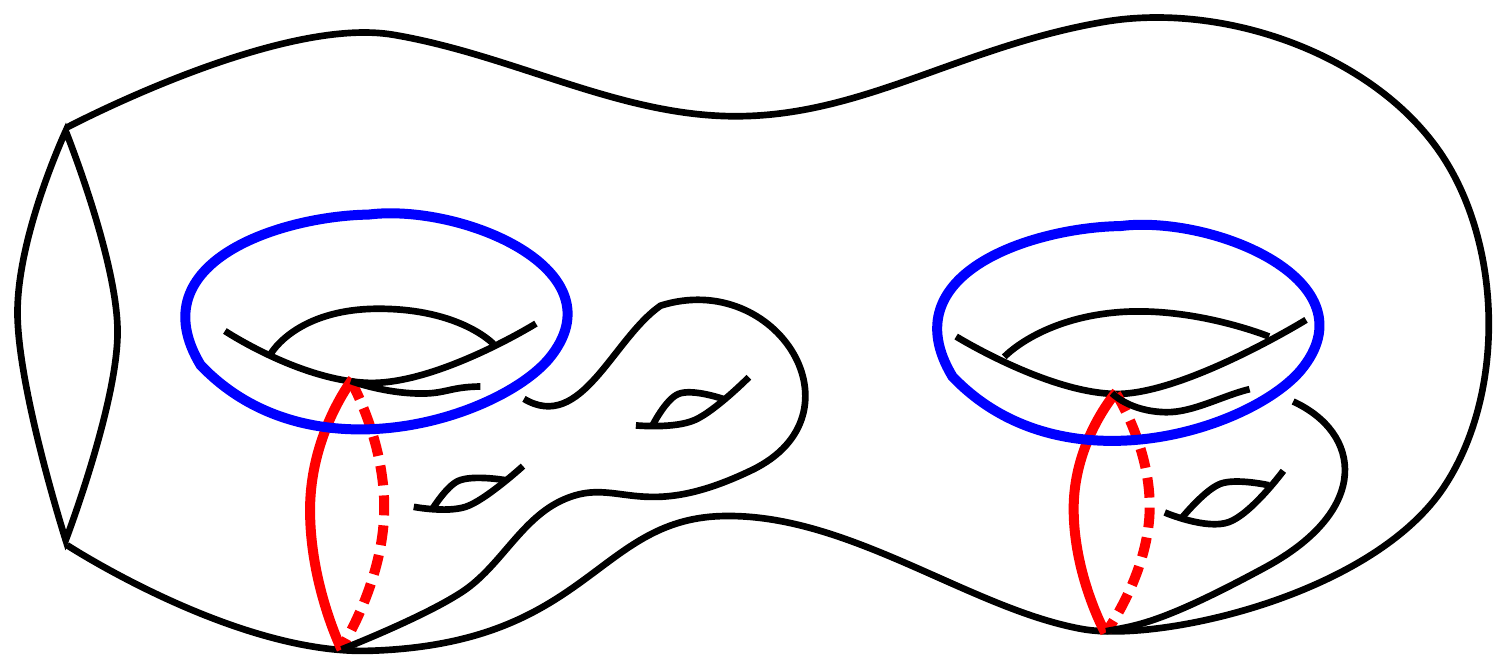}
 \caption{An order 3 half-surface tower.}
\end{figure}

Recall that Theorem \ref{thm:stallings} provided a sufficient condition for all lower central series quotients of a group to be isomorphic. The Dwyer subgroup leads to a refinement of this theorem by Dwyer in \cite{dwyer}.

\begin{thm}[Dwyer's Integral Theorem \cite{dwyer}] \label{thm:dwyerintegral} Let $\varphi: A \rightarrow B$ be a homomorphism that induces an isomorphism on $\text{H}_1( - ; \mathbb{Z} )$. Then for any positive integer $m$ the following are equivalent:

\begin{enumerate}
\item $\varphi$ induces an isomorphism $A/A_m \cong B/B_m$
\item $\varphi$ induces an epimorphism from $\text{H}_2(A;\mathbb{Z})/\Phi_m(A)$ to $\text{H}_2(B;\mathbb{Z})/\Phi_m(B)$
\item $\varphi$ induces an isomorphism from $\text{H}_2(A;\mathbb{Z})/\Phi_m(A)$ to  $\text{H}_2(B;\mathbb{Z})/\Phi_m(B)$ and a monomorphism from $\text{H}_2(A;\mathbb{Z})/\Phi_{m+1}(A)$ to  $\text{H}_2(B;\mathbb{Z})/\Phi_{m+1}(B)$
\end{enumerate}
\end{thm}

In Section \ref{section:background}, we saw Milnor's invariants gave us invariants for a link $L \subset S^3$ which detected exactly when the nilpotent quotients of $\pi_1(S^3 \setminus \nu(L), *)$ were isomorphic to the nilpotent quotients of a free group with appropriate rank. Now, we can use Theorem \ref{thm:dwyerintegral} to construct a concordance invariant for a knot $K \subset M$ where $M$ is a closed, oriented $3$-manifold using similar principles.

\begin{lem}\label{computeinhomology} Let $M$ be a $3$-manifold and $G$ be its fundamental group. Then,
$$\frac{\text{H}_2(M)}{\Phi_m(M)} \cong \frac{\text{H}_2(G)}{\Phi_m(G)}$$ for all $m > 1$.

\end{lem}

\begin{proof}
Recall that we can construct a $K(G,1)$ by attaching cells of dimension $3$ and higher to $M$. Furthermore, the homology groups of $G$ are the homology groups of this $K(G,1)$, thus, we have the following exact sequence

$$\pi_2(M) \rightarrow H_2(M) \rightarrow H_2(G) \rightarrow 0.$$

Now, consider the Dwyer subgroup $\Phi_m(M) \subset H_2(M)$ and notice that every element of $\pi_2(M)$ maps to an element inside $\Phi_m(M)$ since a $2$-sphere is a half-surface tower of order $m$ of arbitrary large class. Therefore, we have the following commutative diagram with exact rows and columns
\begin{center}
\begin{tikzpicture}
\matrix(a)[matrix of math nodes,
row sep=3em, column sep=3em,
text height=1.5ex, text depth=0.25ex]
{& \pi_2(M ) & & & \\
1  &  \Phi_m(M)  & H_2(M) & \frac{H_2(M)}{\Phi_m(M )} & 1 \\
1 & \Phi_m(G) & H_2(G) & \frac{H_2(G)}{\Phi_m(G)} & 1 \\
& 1 & 1 & &  \\};
\path[->](a-1-2) edge (a-2-2);
\path[->](a-1-2) edge (a-2-3);
\path[->](a-2-1) edge (a-2-2);
\path[->](a-2-2) edge (a-2-3);
\path[->](a-2-3) edge (a-2-4);
\path[->](a-2-4) edge (a-2-5);
\path[->](a-3-1) edge (a-3-2);
\path[->](a-3-2) edge (a-3-3);
\path[->](a-3-3) edge (a-3-4);
\path[->](a-3-4) edge (a-3-5);
\path[->](a-2-2) edge (a-3-2);
\path[->](a-3-2) edge (a-4-2);
\path[->](a-2-3) edge (a-3-3);
\path[->](a-3-3) edge (a-4-3);

\end{tikzpicture}
\end{center}

and thus 
$$\frac{H_2(G)}{\Phi_m(G)} \cong \frac{H_2(M) / \text{Im}(\pi_2(M))}{\Phi_m(M )/ \text{Im}(\pi_2(M))}\cong \frac{H_2(M)}{\Phi_m(M)}$$
by the first and third isomorphism theorems. \end{proof}

We can use this lemma to prove Proposition \ref{prop:hcob}.

\begin{proof}[Proof of Proposition \ref{prop:hcob}] Let $C$ be the smooth concordance from $K_1$ to $K_2$ inside $M \times I$, $H=\pi_1((M \times I)\setminus C, *)$, and $\iota_i: M \setminus \nu(K_i) \rightarrow M \times I \setminus \nu(C)$ be inclusion maps. It is a straightforward exercise using the Mayer-Vietoris sequence, naturality, and the 5-lemma to show that each $\iota_i$ induces isomorphisms on homology which, by the proof of Lemma \ref{computeinhomology} and naturality, extend to isomorphisms on the homology of their fundamental groups. Now by applying \ref{thm:dwyerintegral} we see $G(K_1)/G(K_1)_m \cong G(K_2)/G(K_2)_m$.
\end{proof}

Now, given the utility of the Dwyer subgroups in transforming computations in group homology into computations in the homology of a $3$-manifold with appropriate fundamental group, we define a knot invariant. 

\begin{definition}
Let $M$ be a oriented, closed $3$-manifold and $\gamma$ be a fixed smooth, embedded curve inside $M$ such that, for any  knot $K$ homologous to $\gamma$, there is a homomorphism $\psi$ from $G(\gamma)$ to $G(K)$ inducing an isomorphism from $H_1(M \setminus \nu(\gamma))$ to $H_1(M \setminus \nu(K))$.  Then the Dwyer number of such a $K$ relative to $\gamma$ is defined by first considering a fixed $\psi$:
$$D(K, \gamma, \psi) :=  \max\bigg\{ m \bigg\lvert\frac{\text{H}_2(M \setminus \nu(K))}{\Phi_m(M \setminus \nu(K))} \cong \frac{\text{H}_2(M \setminus \nu(\gamma))}{\Phi_m(M \setminus \nu(\gamma))} \text{ induced by $\psi$} \bigg\}$$

then maximize over all such $\psi$:
$$D(K, \gamma) := \max\big\{ D(K, \gamma, \psi) \mid \psi \in \text{Hom}(G(\gamma), G(K)) \text{ where } \psi_* \text{ is an isomorphism on } H_1 \big\}$$

In the case where this set has no maximum, define $D(K, \gamma):=\infty$. 

\end{definition}

The fact that this invariant is never the maximum of the empty set (and therefore vacuous) is not obvious. Recall that $$\Phi_1(\pi_1(X)) = \{x \in H_2(X) \mid x \text{ is a } G_2\text{-surface} \}=H_2(X)$$ by definition, and thus 

$$\frac{\text{H}_2(M \setminus \nu(K))}{\Phi_1(M \setminus \nu(K))} \cong 0 \cong \frac{\text{H}_2(M \setminus \nu(\gamma))}{\Phi_1(M \setminus \nu(\gamma))} $$

From this we see the invariant as at least 1. We can also reformulate $D(K, \gamma)$ in terms of the fundamental group in the following way.

\begin{prop}\label{prop:equivdef}The Dwyer number can be equivalently formulated as $$D(K, \gamma) =\max \bigg\{  m \hspace{2px} \bigg\lvert \hspace{2px} \frac{G(K)}{G(K)_m} \cong \frac{G(\gamma)}{G(\gamma)_m} \text{induced by $\psi$} \bigg\}$$
then maximize over all such $\psi$:
$$D(K, \gamma) := \max\big\{ D(K, \gamma, \psi) \mid \psi \in \text{Hom}(G(\gamma), G(K)) \text{ where } \psi_* \text{ is an isomorphism on } H_1 \big\}$$

In the case where this set has no maximum, define $D(K, \gamma):=\infty$. 
\end{prop}
\begin{proof} By Theorem \ref{thm:dwyerintegral}, a homomorphism $\psi: G(\gamma)\rightarrow G(K)$ inducing an isomorphism on $H_1$ induces an isomorphism on  $\frac{G(K)}{G(K)_m} \cong \frac{G(\gamma)}{G(\gamma)_m}$ if and only if it induces an isomorphism on $\frac{\text{H}_2(G(K))}{\Phi_m(G(K))} \cong \frac{\text{H}_2(G(\gamma))}{\Phi_m(G(\gamma))}$.
We see that Lemma \ref{computeinhomology} implies $\frac{\text{H}_2(G(K))}{\Phi_m(G(K))} \cong \frac{\text{H}_2(G(\gamma))}{\Phi_m(G(\gamma))}$ if and only if  $\frac{\text{H}_2(M \setminus \nu(K))}{\Phi_m(M \setminus \nu(K))} \cong \frac{\text{H}_2(M \setminus \nu(\gamma))}{\Phi_m(M \setminus \nu(\gamma))}$. 
\end{proof}

The utility of the surface tower approach is that it allows us to do computations involving the homology of the fundamental group of a knot complement using iterated surfaces and to geometrically realize these computations. Note that one could generalize this invariant quickly in the following ways:
\begin{enumerate}
\item Define the Dwyer number of a link in $M$,
\item Define the Dwyer number with rational coefficients using the rational lower central series defined by Cochran and Harvey in  \cite{cochranharvey}, which may perhaps be useful in the case of $M$ whose fundamental group is not torsion-free. 
\end{enumerate} These questions are beyond the scope of this paper, but it would be interesting to see computations of the Dwyer number using distinct tools from this work.  Note that by arguments in Section \ref{section:background}, it follows immediately that the Dwyer number of a link $L \subset S^3$ is exactly the weight of the first non-vanishing Milnor's invariant of $L$.

\concinvt

\begin{proof} Let $\gamma$ be a fixed smooth, embedded curve inside $M$ such that, for any  knot $K$ homologous to $\gamma$, there is a homomorphism from $G(\gamma)$ to $G(K)$ inducing an isomorphism from $H_1(M \setminus \nu(\gamma))$ to $H_1(M \setminus \nu(K))$. 

If $J$ is another oriented knot in $M$ in the same concordance class as $K$ with fundamental group $G(J)$, it follows immediately that $K$ and $J$ are homologous in $M$ and therefore $D(J, \gamma)$ is well-defined. By Proposition \ref{prop:hcob}, $G(K)/G(K)_m \cong G(J)/ G(J)_m$ for all $m$ since $K$ and $J$ are concordant. The result follows from the proof of Proposition \ref{prop:equivdef}.

\end{proof}

\section{A generalization of the first non-vanishing Milnor's invariant} \label{section:dwyernum}

In the case where $K$ is null-homologous and $M \cong \inlineMl$, this invariant is straightforward to work with and has direct connections to Milnor's invariants for links in $S^3$. In this paper, we will focus on this case as it has applications to the study of link concordance in $S^3$ as we will mention in Section \ref{section:dwyernum}; we will address other cases in future work. 

\begin{lem}Let $K$ be a null-homologous knot in $\inlineMl$, and let $U$ be the unknot in $\inlineMl$. Then, $U$ can be used as $\gamma$ in the definition of $D(K, \gamma)$. 

\end{lem}

\begin{proof}
A simple Seifert-van Kampen argument shows $G(U)$ is a free group $F$ on $l+1$ generators. Furthermore, a similar argument shows that $G(K)$ is normally generated by a meridian of $K$ and one meridian for each $S^1 \times S^2$ factor of $\inlineMl$. It follows immediately that the map $\varphi: G(U) \rightarrow G(K)$ sending generators of $\Gamma$ to the aformentioned meridians induces an isomorphism on first homology.
\end{proof}

As a result, in the case when $K$ is null-homologous in $\inlineMl$ we will drop $\gamma$ from the notation and simply write $D(K)$. 

\begin{cor}
Let $K \subset \inlineMl$ be a null-homologous knot. Then,

$$D(K) = \max\{\hspace{3px} q\hspace{7px} |\hspace{7px} \frac{\text{H}_2(\inlineMl \setminus K)}{\Phi_q(\inlineMl \setminus K)} = 0\hspace{3px}\}$$
\end{cor}
\begin{proof}
Since $H_2(F)=0$, the corollary follows. 
\end{proof}

Recall that we can view Milnor's invariants for links $L \subset S^3$ as detecting how deep the longitudes of $L$ lie in the lower central series of the link group. In a similar fashion, the Dwyer number for null-homologous knots in $\inlineMl$ detects how deep the $0$-framed longitude of $K \subset \inlineMl$ is in the lower central series of the knot group (and thus how ``close" the group quotients of the knot group are to the group quotients of the corresponding free group). Similarly, it also detects the length of the first non-zero Massey product in the cohomology of the complement of the knot $K$ in an analogous way to the celebrated theorem of Turaev \cite{turaev} and independently Porter \cite{porter} relating Milnor's invariants of links in $S^3$ to Massey products in the link complement. The case of $K \subset \inlineMl$ may serve as a model for the computation of $D(K, \gamma)$ in more general situations as we will explore in later work.

\begin{figure}
\centering
  \includegraphics[width=.2\linewidth]{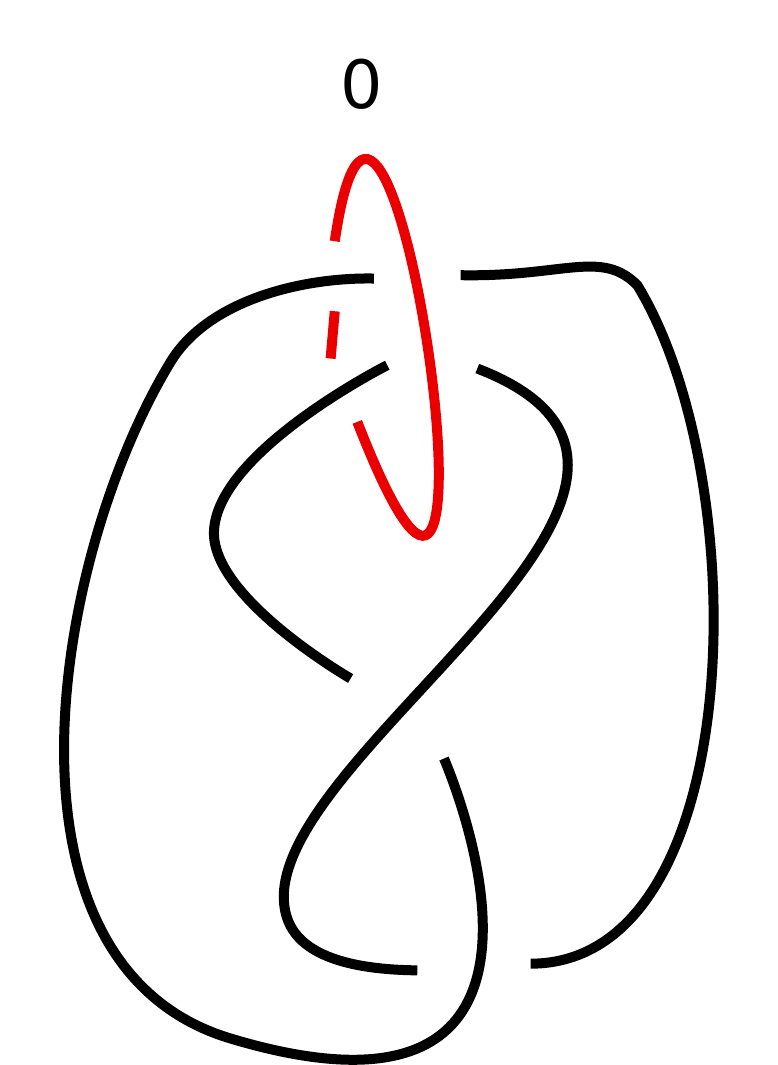}
  \caption{A null-homologous knot in $S^1 \times S^2$.}
  \label{fig:dwyerexample}
\end{figure}

To compute $D(K)$ we find the following lemmas useful.

\begin{lem} \label{prop:mv}
Let $K$ be a null-homologous knot in $\inlineMl$ and $\nu(K)$ be a tubular neighborhood of it. Then 

\begin{align*}
    \text{H}_0(\Ml \setminus \nu(K)) &= \mathbb{Z}        \\ 
     \text{H}_1(\Ml \setminus \nu(K)) &= \mathbb{Z}^{\ell+1} \\  
      \text{H}_2(\Ml \setminus \nu(K)) &= \mathbb{Z}^{\ell} \\
      \text{H}_n(\Ml \setminus \nu(K)) &= 0 \text{ for } n \geq 3
\end{align*}

 Moreover, $\text{H}_1(\inlineMl \setminus \nu(K)) = \langle \mu, d_1, ..., d_{\ell} \rangle$ where $\mu$ is a meridian of $K$ and $d_1, ..., d_l$ generate $\text{H}_1(\Ml)$.  $\text{H}_2(\Ml \setminus \nu(K)) = \mathbb{Z}^{\ell}$ is generated by surfaces consisting of the $S^2$ generators of $\text{H}_2(\inlineMl)$ with attached tubes missing $K$.
\end{lem}
\begin{proof}
First note that any $K \subset \inlineMl$ can be described by an obvious surgery diagram. We can clearly find an $\ell+1$-component link $L = (K^{\prime}, U_1, ..., U_{\ell}) \subset S^3$ such that the $\ell$-component sublink $U = (U_1, ..., U_{\ell}) $ is the unlink and after performing $0$-surgery on $U$ the image of $K^{\prime}$ in the resulting $\inlineMl$ is $K$. Consider a Mayer-Vietoris sequence in reduced homology with neighborhoods $A = \nu_{\epsilon}(K)$ of radius $\epsilon >0$ for some $\epsilon \in \R$ and $B=\inlineMl \setminus \nu_{\epsilon/2}(K)$ of radius $\epsilon/2 >0$. Then $ A \cup B = \inlineMl$ and $A \cap B$ deformation retracts to $\partial\nu(K)$. Note that $\text{H}_1(\inlineMl) = \mathbb{Z}^{l}$ is generated by meridians $d_i$, $1 \leq i \leq \ell$ of unlink $U$. $\text{H}_2(\inlineMl) = \mathbb{Z}^{\ell}$ is generated by embedded 2-spheres $e_i$, $1 \leq i \leq \ell$ such that each sphere $e_i$ is the union of the disk $\Sigma_i$ bounded by $U_i$ in $S^3$ and a disk $D_i$ bounded by $U_i$ after $0$-surgery on $U$.

We will compute using cellular homology. There is a cellular decomposition of $\inlineMl$ induced by the surgery diagram: start with a cellular decomposition for $S^3 \setminus \nu(u)$, and then performing $0$-surgery on $U$ glues a $2$-cell to the longitude of $U_i$ and a $3$-cell to the resulting $S^2$ boundary component. Therefore we see that $\partial_*: \text{H}_3(\inlineMl) \rightarrow \text{H}_2(\partial\nu(K))$ is an isomorphism from the definition of the boundary map, so we have the following long exact sequence:
$$ 0 \rightarrow \tilde{\text{H}}_2(\inlineMl \setminus \nu(K)) \rightarrow \mathbb{Z}^l \rightarrow \mathbb{Z}^2 \rightarrow \mathbb{Z} \oplus \tilde{\text{H}}_1(\inlineMl \setminus \nu(K)) \rightarrow \mathbb{Z}^l \rightarrow 0 $$
Since the last term is a projective module, there is a splitting map $s: \mathbb{Z}^{\ell} \rightarrow \mathbb{Z} \oplus \tilde{\text{H}}_1(\inlineMl \setminus \nu(K))$ and thus by exactness $\mathbb{Z} \oplus \tilde{\text{H}}_1(\inlineMl \setminus \nu(K)) \cong \mathbb{Z}^l \oplus im(f: \mathbb{Z}^2 \rightarrow \mathbb{Z} \oplus \tilde{\text{H}}_1(\inlineMl \setminus \nu(K)))$ where $f$ sends meridian $\mu$ to $(0, -\jmath_*(\mu))$  and longitude $\lambda$ to $(\imath_*(\lambda),0)$ where $\jmath$ and $\imath$ are inclusion maps since $K$ is null-homologous. Hence $\mathbb{Z} \oplus \tilde{\text{H}}_1(\inlineMl \setminus \nu(K)) \cong \mathbb{Z}^l \oplus \mathbb{Z} \oplus \langle-\jmath_*(\mu)\rangle \cong \mathbb{Z}^{l+1} \oplus \langle-\jmath_*(\mu)\rangle$. 

Now we should determine the order $d$ of $-\jmath_*(\mu) \in \tilde{\text{H}}_1(\inlineMl \setminus \nu(K))$. If $d (-\jmath_*(\mu)) = 0$, then $\text{Ker }(\tilde{\text{H}}_1(\partial\nu(K)) \rightarrow \mathbb{Z} \oplus \tilde{\text{H}}_1(\inlineMl \setminus \nu(K))$, which is equal to $d\mathbb{Z}=\text{Im }(\tilde{\text{H}}_2(\inlineMl )\cong \mathbb{Z}^{\ell} \rightarrow H_1(\nu(K))\cong \mathbb{Z}^2)$.

Since $K$ is null-homologous, viewing $K^{\prime}$ as a 1-component sublink of $L$ as above gives $lk(K^{\prime}, U) = 0$ which implies that the algebraic intersection $ K \cdot \Sigma_i = 0$. Furthermore,  using the aforementioned surgery diagram  of $K \subset \inlineMl$, it is clear that the geometric intersection $K \cap D_i = 0$. Thus, $\partial\nu(K)$ intersects sphere $e_i=\Sigma_i \cup_{L_i} D_i$ in pairs of oppositely oriented meridians of $K$ (or does not intersect $e_i$ at all). This means that for each generating sphere $e_i$ of $H_2(\inlineMl)$, we can compute its image under $\partial_*$ in the Mayer Vietoris sequence by the boundary of meridional disks in $\nu(K)$ which, by the previous argument, must be algebraically 0. Therefore $\partial_*[e_i]=0$ for each $i$ and $\text{Im }(\mathbb{Z}^{\ell} \rightarrow \mathbb{Z}^2) = 0$, so $-\jmath_*(\mu)$ is infinite order. 

We have therefore shown $\mathbb{Z} \oplus \tilde{\text{H}}_1(\inlineMl \setminus \nu(K)) \cong \mathbb{Z}^{\ell+2}$ and thus $\tilde{\text{H}}_1(\inlineMl  \setminus \nu(K)) \cong \mathbb{Z}^{\ell+1}$. Moreover, by the previous argument about intersections we see that we can take as a generating set the union of each $e_i$ with $i_j$ tubes $T_{i_j}$, one for each oppositely oriented pair of intersections $p_{i_j}$ and $q_{i_j}$ between $K$ and $e_i$. Finally, by exactness we see that $\text{im }(\tilde{\text{H}}_2(\inlineMl \setminus \nu(K)) \rightarrow \mathbb{Z}^{\ell}) = \mathbb{Z}^{\ell}$ and therefore $\tilde{\text{H}}_2(\inlineMl \setminus \nu(K)) \cong \mathbb{Z}^{\ell}$.
\end{proof}

Recall from Section \ref{section:background} that, in order to define Milnor's invariants, we exploited a specific presentation of the nilpotent quotients of the fundamental group of the link complement  in $S^3$. We obtain a similar presentation in this context.

\begin{lem} \label{gppresent} Let $K \subset \inlineMl$ be a null-homologous knot with longitude $l$. There is a homomorphism $\varphi: F \rightarrow G(K,\inlineMl)$ where $F$ is the free group on $\ell+1$ letters and a word $R_m(l) \in F$ with $\varphi(R_m(l)) \cong l \text{ mod } G(K)_m$ such that $G(K) / G(K)_m$ can be presented by:
$$ \langle x, a_1, ..., a_{\ell} \mid [x, R_m(l)], F_m \rangle$$ 
where $F_m$ are the weight $m$ commutator relations and each $a_i$ generates the $i^{\text{th}}$ $\Z$ summand in $H_1(\inlineMl)$.
\end{lem}
\begin{proof}
Notice that $K \subset \#^{m}S^1 \times S^2$ can be viewed as the result of $0$-surgery on the trivial sublink $U =(U_1, U_2, ..., U_m)$ of the ordered, oriented $m+1$ component link $L = (K^{\prime}, U_1, U_2, ..., U_m) \subset S^3$.  By a Theorem 4 in \cite{milnor57} the quotients $G(L,S^3)/ G(L,S^3)_q$ have the presentation 
$$\langle x, a_1, ..., a_{\ell} \mid [x, R_m(l_1)], [a_1, R_m(l_2)], ..., [a_{\ell}, R_m(l_{\ell+1})], F_m \rangle$$
where $l_i$ is a word representing the $i^{\text{th}}$ longitude of $L$ and, for the meridian homomorphism $\varphi: F \rightarrow G(L,S^3)$, $R_m(l_i)$ is a word such that $\varphi(R_m(l_i)) \cong l_i \text{ mod } G(L,S^3)_m $. Since $K \subset \#^{\ell}S^1 \times S^2$ is obtained from $L \subset S^3$ by surgery along longitudes $l_2, ..., l_{\ell+1}$ we can see 
$$G(K, \inlineMl) \cong G(L,S^3) / \langle l_i \mid 2 \leq i \leq \ell+1\rangle.$$ 

A straightforward Mayer-Vietoris argument shows that viewing $\inlineMl \setminus \nu(L)$ as the result of surgery on $U$ results in a surjection on the fundamental groups $p: G(L,S^3) \twoheadrightarrow G(K, \inlineMl)$ induces surjections $p_m: G(L,S^3) / G(L,S^3)_m \twoheadrightarrow G(K, \inlineMl)/ G(K, \inlineMl)_m$. Since the kernel of $p$ is generated by longitudes $l_i$, $2 \leq i \leq \ell+1$  and each $l_i \cong \varphi (R_m(l_i)) \text{ mod } G(L,S^3)_m$ we see that $p_m(R_m(l_i))=1$ and $G(K, \inlineMl) / G(K, \inlineMl)_m$ is presented by 
$$\langle x, a_1, ..., a_{\ell} \mid [x, R_m(l)], F_m \rangle.$$

\end{proof}

This allows us to quickly obtain the following result, showing that $D(K)$ detects how deep the longitude of $K$ is in the lower central series of the fundamental group of its complement just as Milnor's invariants do.
\longitudethm

\begin{proof}
Consider the homomorphism $\varphi: F = \langle x, a_1, ..., a_{\ell} \rangle \rightarrow G(K)$ sending $x$ to a meridian of $K$ and each $a_i$ to a loop homotopic to the $i^{\text{th}}$ copy of $S^1 \times \{0\}$. This homomorphism clearly induces isomorphisms $H_1(F) \rightarrow H_1(G(K)) \cong \Z^{\ell+1}$. $D(K) = q$ implies $\text{H}_2(\inlineMl \setminus \nu(K)) / \Phi_q(\inlineMl \setminus \nu(K)) = 0$. As in the proof of Proposition \ref{prop:equivdef} we immediately see $F / F_q \rightarrow G(K) / G(K)_q$.

By Lemma \ref{gppresent}, we also know 
$$ \frac{G(K)}{G(K)_q} \cong \langle x, a_1, ..., a_{\ell} \mid [x, R_q(l)], F_q \rangle$$ 
and for it to be isomorphic to $F / F_q$, $[x, R_q(l)]$ must be in $F_q$. Thus, $[x, R_q(l)]$ is trivial in $F / F_q$. Since $x$ is a generator of $F$, this implies $R_q(l) \in F_{q-1}$ and thus is trivial in $F/F_{q-1}$. However, by properties of lower central series quotients we know $F / F_q \cong G(K) / G(K)_q$ implies $F / F_{q-1} \cong G(K)/ G(K)_{q-1}$ so $\varphi(R_q(l))$ is trivial in $G(K) / G(K)_{q-1}$. Recall $\varphi(R_q(l)) = la$ where $a \in G(K)_q$. By the third isomorphism theorem, $G(K) / G(K)_{q-1} \cong (G(K) / G(K)_{q}) / (G(K)_{q-1} / G(K)_{q})$. The image of a longitude $l$ in this quotient is in the same class as $la$ (as they differ by an element of $G(K)_{q}$). Thus, the residue class of a longitude $l \in G(K)$ inside $G(K) / G(K)_{q-1}$ is mapped under this isomorphism to the class represented by $la$ which is trivial. Thus, $l \in G(K)_{q-1}$. Furthermore, since $D(K) < q+1$, $\text{H}_2(\inlineMl \setminus \nu(K)) / \Phi_{q+1}(\inlineMl \setminus \nu(K)) \neq 0$ and by similar arguments, no $R_{q+1}(l)$ is in $F_q$ and therefore $l \not\in G(K)_q$.
\end{proof}

We further see that this Dwyer number can detect the weight of the first non-vanishing Massey product in the complement of null-homologous $K \subset \inlineMl$.

\masseythm

\begin{proof}
This follows directly from an argument in the proof by Cochran, Gerges, and Orr   in Proposition 6.8 in \cite{cochrangergesorr}. More specifically, if $D(K)=q$ then the knot group quotients $G(K)/G(K)_m$ are isomorphic to the free group quotients $F/F_m$ for $m < q$ by \cite{dwyer} and this isomorphism is induced by the meridional map $F \rightarrow G(K)$. This is exactly the group quotient criteria used by \cite{cochrangergesorr} in the context of $k$-surgery equivalent manifolds and we can see that by their proof, all Massey products of weight less than $q$ in $H^*(\inlineMl \setminus \nu(K))$ vanish and there is a Massey product of weight $q$ in this cohomology ring which is nonzero. 
\end{proof}

At first glance the definition of $D(K)$ looks quite unrelated to that of Milnor's invariants defined in Section \ref{section:background}; however, the previous results show $D(K)$ is a concordance invariant which detects how deep a knot's longitude is in the lower central series of its knot group and detects the weight of the first non-vanishing Massey product in the knot complement. These are exactly the properties of Milnor's invariants that make them so useful. This is no coincidence; as we will see, the Dwyer number of a knot in $\inlineMl$ is directly related to the Milnor's invariants of an associated link in $S^3$. We can in fact use the Milnor's invariants of this associated link to compute the Dwyer number. The following lemma lays the groundwork to do this.

\begin{lem} \label{lem:milnorinvtdwyer}
Let $L = (K^{\prime}, U_1, ..., U_{\ell}) \subset S^3$ be an $\ell+1$-component link such that the $\ell$-component sublink $U = (U_1, ..., U_{\ell}) $ is the unlink and 0-surgery on $U$ gives a null-homologous knot $K$ in $\inlineMl$. If all Milnor invariants $\bar{\mu}_L(I)$ of weight $|I| < q$ are trivial, then $D(K) \geq q$. 
\end{lem}
\begin{proof} 
Let $m_1, ..., m_l$ be meridians of $U$. Since $U$ is an unlink, the longitudes $u_i^{\prime}$ of its components $u_i$ bound disjoint disks, call these disks $\{D_i\}$. Once we perform $0$-surgery on $U$, the images of $u_i^{\prime}$ also bound disjoint disks $\Delta_i$ inside the surgery tori. Clearly $\{S_i^2 = D_i \cup \Delta_i\}$ is a generating set for $H_2(\inlineMl)$.

By Lemma 2.4 and Lemma 2.7 of \cite{freedmanteichner}, $\bar{\mu}_L(I) = 0$ for $|I| < q$ implies each longitude $l_i$ of $U$ is the boundary of the image of an order $q-1$ half-surface tower $\Sigma_i$ under a map $f_i : \Sigma_i \rightarrow S^3 \setminus L$. Notice that we can extend each of these maps to $\bar{f}_i: \Sigma_i \cup \Delta_i \rightarrow S_0^3(U) = \inlineMl$. Notice that the images of these order $q-1$ half-surface towers are disjoint from the knot $K$.

Let $F_i$ be the image of the first stage of $\bar{f}_i$. Consider the homology class $[F_i] \in \text{H}_2(\inlineMl)$. We know from intersection theory that 
$$[F_i] = \sum_{j=1}^{\ell}n_j[S_j^2]=\sum_{j=1}^{\ell}\big([F_i]\cdot [m_j] \big)[S_j^2] = [S_i^2]$$

This is because all pairwise linking numbers between components of $L$ are $0$ (since these linking numbers are just the weight $2$ Milnor invariants of L and $K$ is nullhomologous in the result of surgery on $U$), and by construction, $[F_i]\cdot [m_j] = lk(U_i, m_j)= \delta_{ij}$.

Since  $\text{H}_2(\inlineMl \setminus \nu(K)) \cong \mathbb{Z}^{\ell}$ by Lemma \ref{prop:mv} we see now that $\{[F_i]\}$ generates $\text{H}_2(\inlineMl \setminus \nu(K)) $. By construction, $\{[F_i]\}$ generates $\Phi_{q}(\inlineMl \setminus \nu(K))$ also  and therefore $\text{H}_2(\inlineMl \setminus \nu(K)) /\Phi_{q}(\inlineMl \setminus \nu(K)) =0$, so $D(K) \geq q$. 
\end{proof} 

Equipped with this lemma, we can now prove the following theorem allowing us to compute the the Dwyer number of a knot using the Milnor's invariants of an associated link. 

\begin{thm} \label{thm:computationthm}
Let $L = (K^{\prime} , U_1, ..., U_{\ell})$ be an ordered, oriented link in $S^3$ such that the sublink $U = (U_1, ..., U_{\ell})$ is an unlink and $0$-surgery on $U$ results in a null-homologous knot $K\subset \#^{\ell} S^1 \times S^2$.
If $\bar{\mu}_L(I) = 0$ for $|I| < q$ and, for some multi-index $J = (i_1, ... , i_q)$ of length $q$, $\bar{\mu}_L(J) \neq 0$, then $D(K) = q$.
\end{thm}
\begin{proof} Note that since $U$ is an unlink, a non-vanishing Milnor invariant of $L$ must involve the index $1$. Since all invariants of weight less than q vanish, $i_k=1$ for some $k$ and therefore $ \bar{\mu}_L(J)$ is equal to $\bar{\mu}_L(J^{\prime} 1))$ where $J^{\prime} 1$ is a cyclic permutation of $i_1, ., i_k, i_{k+1}.. , i_q$ by Theorem 10.14 in Hillman's work in \cite{hillman}.  Notice since $\inlineMl$ is constructed via attaching $2$ and $3$ cells to $S^3 \setminus \nu(L)$, we have a surjection $\varphi: G(L,S^3) \rightarrow G(K,\inlineMl)$ whose kernel is generated by the longitudes of $U$.

Let $F = \langle x_1, ..., x_{\ell+1} \rangle$ be a free group. By our condition on the Milnor invariants of $L$, a result of Milnor \cite{milnor57} shows the map $h: F \rightarrow G(L,S^3)$ sending $x_i$ to a meridian of the $i^{th}$ component of $L$ induces an isomorphism $F / F_q \cong G(L,S^3) /G(L,S^3)_q$. As shown in Lemma \ref{lem:milnorinvtdwyer}, our assumption also implies $D(K^{\prime}) \geq q$. Now, the definition of $D(K)$ combined with Dwyer's theorem then gives us that the map $g: F \rightarrow G(K,\inlineMl)$ sending $x_1$ to the meridian of $K^{\prime}$ and $x_i$ to the image under $\varphi$ of a meridian of $U_{i-1}$ for $ 2 \leq i \leq \ell+1$ induces an isomorphism $F / F_{q} \cong G(K,\inlineMl)/ G(K,\inlineMl)_{q}$.

The conditions on $\bar{\mu}_L(I)$ give us that a longitude $\lambda_{K^{\prime}}$ of $K^{\prime}$ does not lie in $G(L, S^3)_{q}$. In other words, such a class is nontrivial in $G(L,S^3) /G(L,S^3)_{q}$. By properties of the lower central series and by construction, we have the following commutative diagram.

\begin{center}
\begin{tikzcd}
G(L,S^3) / G(L,S^3)_{q}   \arrow[rr, "\varphi_{q}"] & & G(K,\inlineMl) / G(K,\inlineMl)_{q}  \arrow[dl, "g_{q}"]   \\
 & F / F_{q} \arrow[ul, "h_{q}"] &
\end{tikzcd}
\end{center}

We see that $\varphi(\lambda_{K^{\prime}})$ is conjugate to a longitude of $K$. From our diagram, we see $G(K,\inlineMl) / G(K,\inlineMl)_{q} \cong G(L,S^3)/ G(L,S^3)_{q}$ by the isomorphism $g_{q} \circ h_{q}$ and thus $\varphi_q(\lambda_{K^{\prime}})$ is nontrivial in $G(K,\inlineMl) / G(K,\inlineMl)_{q-1}$. 
Since $G(K,\inlineMl)_{q-1}$ is normal, this means no possible basing of a $0$-framed longitude of $K^{\prime}$ is in $G(K,\inlineMl)_{q-1}$ and by Theorem \ref{thm:longitude}, $D(K) < q+1$.
\end{proof}

We see this theorem gives us a way to construct many null-homologous knots which are not concordant to each other (or the unknot) using various realization theorems for Milnor's invariants. In particular, the beautiful examples involving Bing doubling along a tree by Tim Cochran in \cite{cochranmemoir} allow us to construct links in $S^3$ whose only non-zero Milnor's invariants are a specific weight $q$ (and he further outlines procedures for finding more general ``antiderivatives" of links given a specific iterated commutator of generators.) It is important to note that if a knot $K \subset \#^n S^1 \times S^2$ has an associated $n+1$ component link $L$ as in Theorem \ref{thm:computationthm} which is not link homotopic to the trivial link, the Dwyer number is perhaps a more subtle invariant than is needed as the homotopy class of the knot in the $3$-manifold is non-trivial. Therefore, to construct examples illustrating the utility of $D(K)$ to show a knot is not concordant to the unknot, they should correspond to links in $S^3$ which are link homotopic to the trivial link. We can construct such examples using the family of ``Sydney's links" defined in \cite{cochranmemoir} and denoted $L_{2n}$ where $n \geq 5$.

\begin{figure}[h] 
\centering
  \includegraphics[width=.7\linewidth]{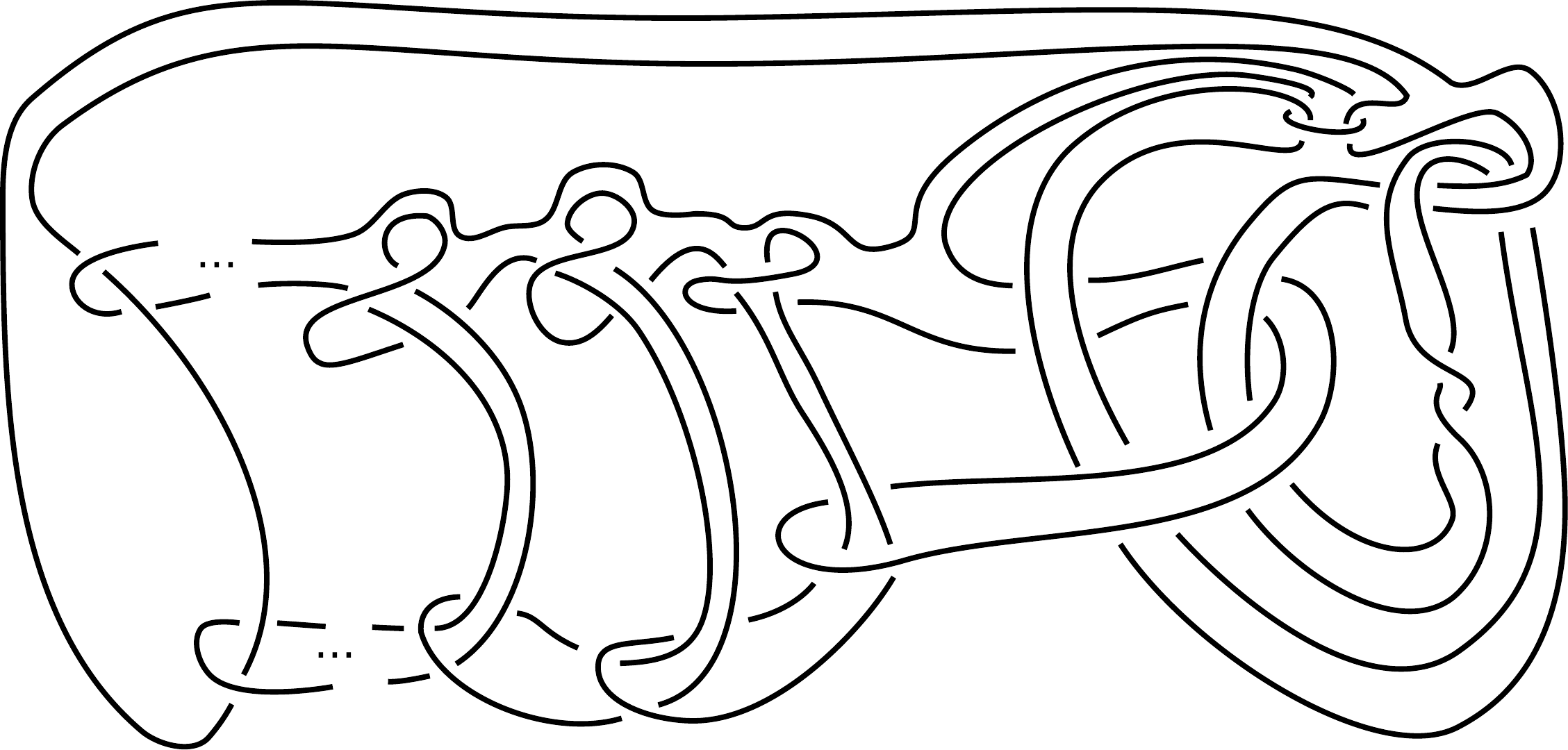}
\caption{A family of two-component links $L_{2n}$, $n \geq 5$ which are homotopically unlinked, but have Milnor invariant $\bar{\mu}(1...12211221...1)=(-1)^{n+1}$ where each string $1...1$ has $1$ $n-3$ times. The ellipses denote $n-3$ repetitions of the repeated tangle.}
\label{fig:L2n}
\end{figure}

\slicenotconc

\begin{proof}
Note that $L \setminus U$ is a slice knot in $S^3$, and therefore it bounds a disk $\Delta$ in $B^4$. Denote the result of attaching $0$-framed $2$-handles to $B^4$ along each of the $n$ components of $U$ as $X$. Since $\Delta$ is properly embedded, we immediately see the image of $L \setminus U$ is slice. 

Now, assume $L$ has some lowest weight non-vanishing Milnor invariant $\bar{\mu}_L(I)$ of weight $q$. Then $D(K) = q$ by Theorem \ref{thm:computationthm} and therefore by Theorem \ref{thm:concinvt} $K$ is not concordant to the unknot.
\end{proof}

\realization
\begin{proof}
We see that Sydney's links with $0$-surgery performed on the component in Figure satisfy the hypothesis of Proposition \ref{thm:slicenotconc}. More precisely, they give a family of knots $K_{2n}$ in $S^1 \times S^2$ which are null-homotopic, slice in $D^2 \times S^2$, and which are not concordant to the unknot in $S^1 \times S^2 \times I$. 

\begin{figure}[h] 
\centering
  \includegraphics[width=.7\linewidth]{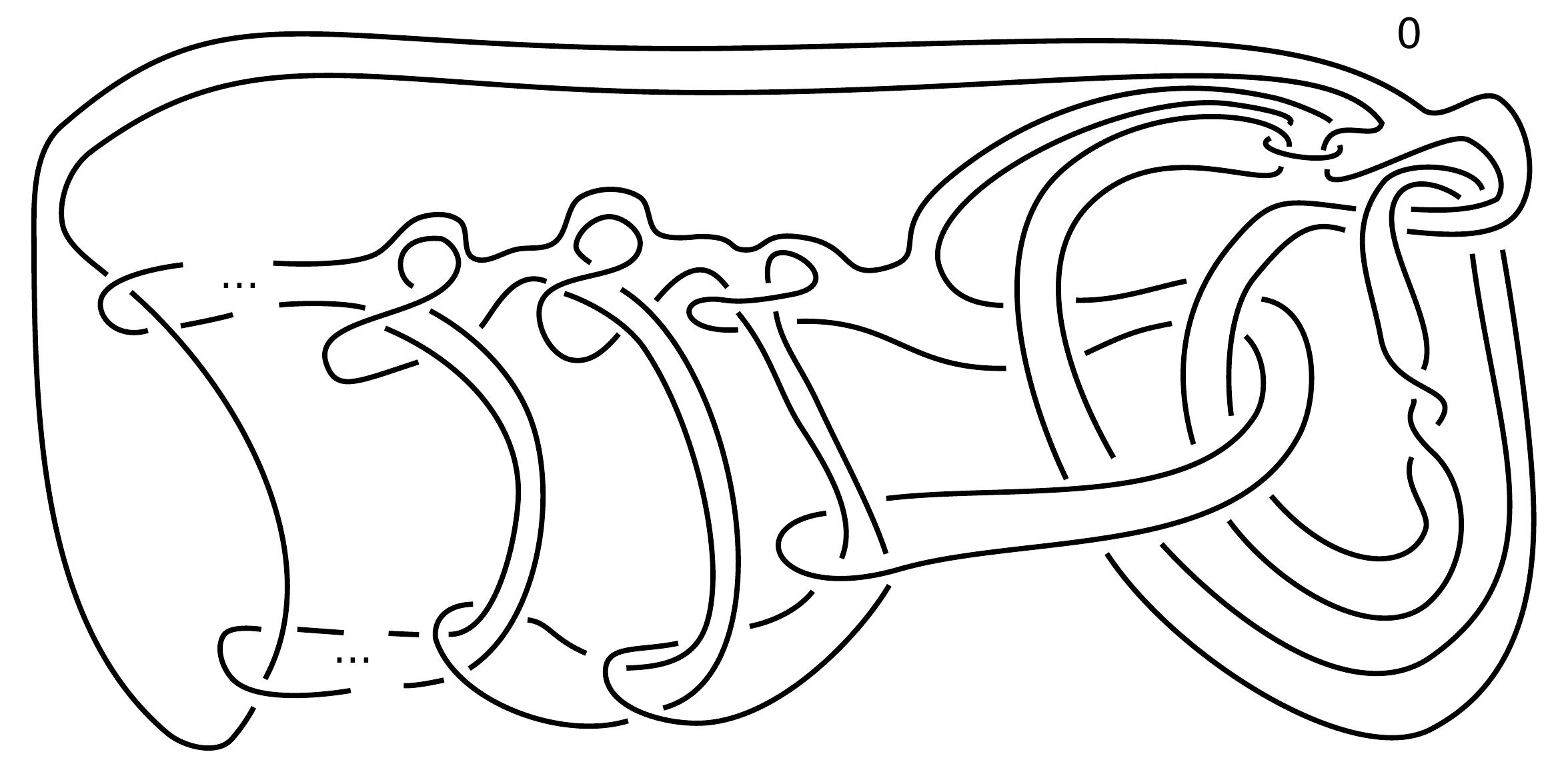}
\caption{The result of $0$-surgery on one component of $L_{2n}$. By Theorem \ref{thm:computationthm}, these knots have Dwyer number $2n$.}
\label{fig:sydsurginonecopy}
\end{figure}

Note that it is straightforward to construct links in further numbers of copies of $S^1 \times S^2$ which are null-homotopic, slice in boundary connected sums of $D^2 \times S^2$, and which are not concordant to the unknot in $S^1 \times S^2 \times I$. Let $L^{\prime}_{2n}$ be the l with first two components $L_{2n}$ labeled $x$ and $y$ with an additional $n-5$ components labeled $z_1$ through $z_{n-5}$ as pictured in \ref{fig:sydsurgmod}. Using the same surfaces for the first two components as in \cite{cochranmemoir} and additional punctured tori for each new component, we can compute that the lowest non-vanishing Milnor invariant is weight $4$. From the diagram, we see the surfaces for $x$ and $y$ can be perturbed to intersect in the connected curve $c(xy)$ and the surfaces for the additional components intersect the surface for $x$ as shown. Since $c(xy)$ and $c(xz_1)$ are weight two curves (using the terminology of Cochran) which have non-zero linking number, and one can check from the diagram there are no non-zero linkings of lower weight curves, the link in question has a non-vanishing Milnor invariant of weight $4$. Therefore, the result of $0$-surgery on the same component of the $L_{2n}$ sublink and the additional unlinked components has Dwyer number $4$ by Theorem \ref{thm:computationthm}.\end{proof}

\begin{figure}[h] 
\centering
  \includegraphics[width=\linewidth]{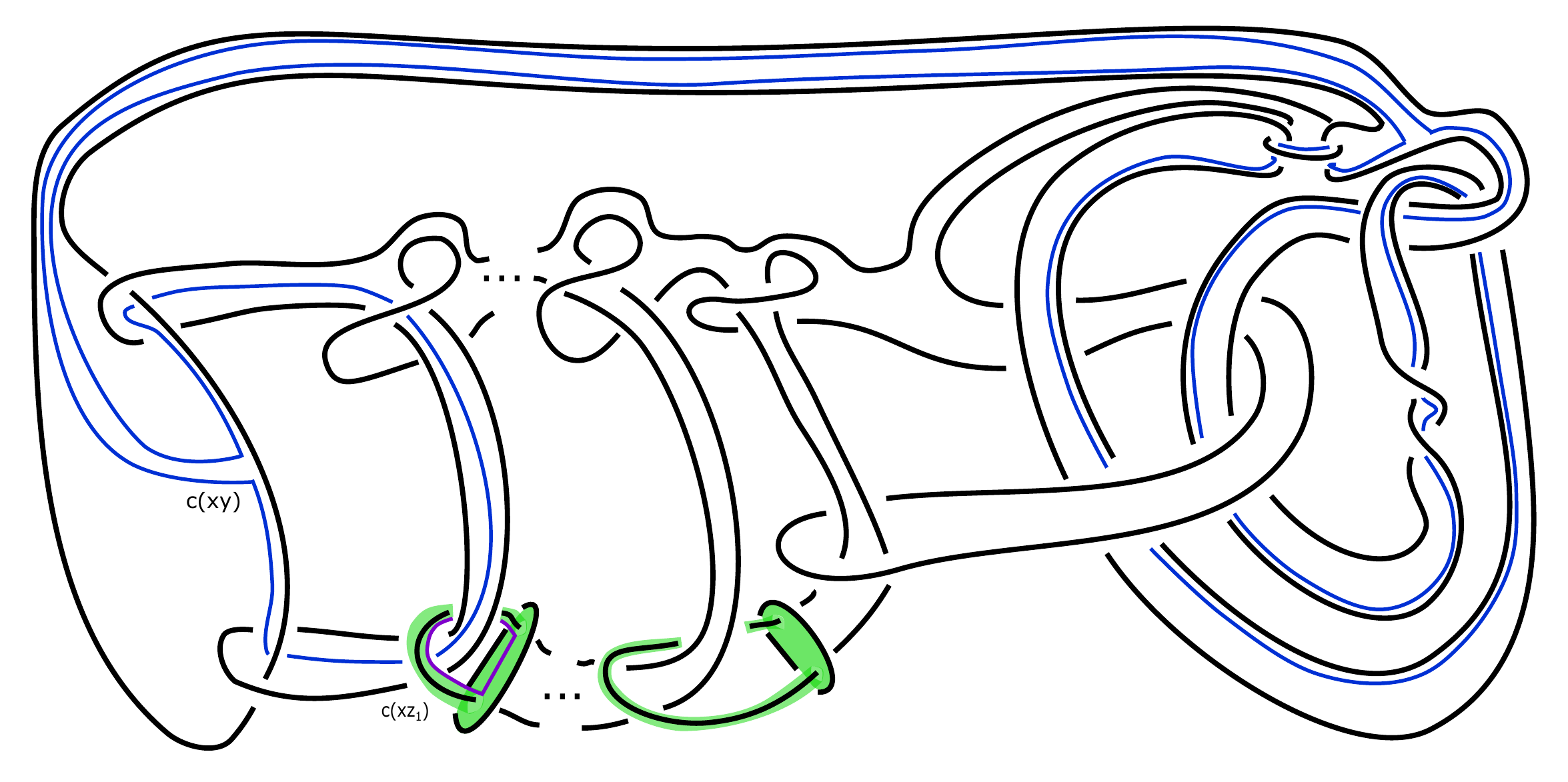}
\caption{The links $L^{\prime}_{2n}$ along with surface intersection curves $c(xy)$ and $c(xz_1)$.}
\label{fig:sydsurgmod}
\end{figure}

Note that it is often possible to get a lower bound on the Dwyer number directly by constructing embedded half-surface towers of order $m$ which generate the second homology of the complement. In this way, if we have a fixed $K$ for which $D(K)$ is known, we can sometimes directly show a specific $J$ is not concordant to it by demonstrating the appropriate half-surface tower of order $m$ where half of a symplectic basis for the surface bounds further embedded surfaces as in Figure \ref{fig:class6}. Pictured is a knot in $J \subset \#^2 S^1 \times S^2$ with an embedded half $(6)$-surface and an embedded $S^2$, together they generate $H_2(\#^2 S^1 \times S^2 \setminus K)$ and thus $J$ is not concordant to $L_2$ as $D(L_2)=4$ but $D(J) \geq 5$. Note that this bound is not sharp as a $G_m$-surface is not exactly the same thing as an embedded surface tower; from Theorem \ref{thm:computationthm} and Section \ref{section:background} we can compute that the lowest weight non-vanishing Milnor's invariant underlying link in Figure \ref{fig:L2n} is $\bar{\mu}_L(111111111112)=1$ and thus $D(J) = 12$. 

\begin{figure}[h] 
\centering
  \includegraphics[width=.7\linewidth]{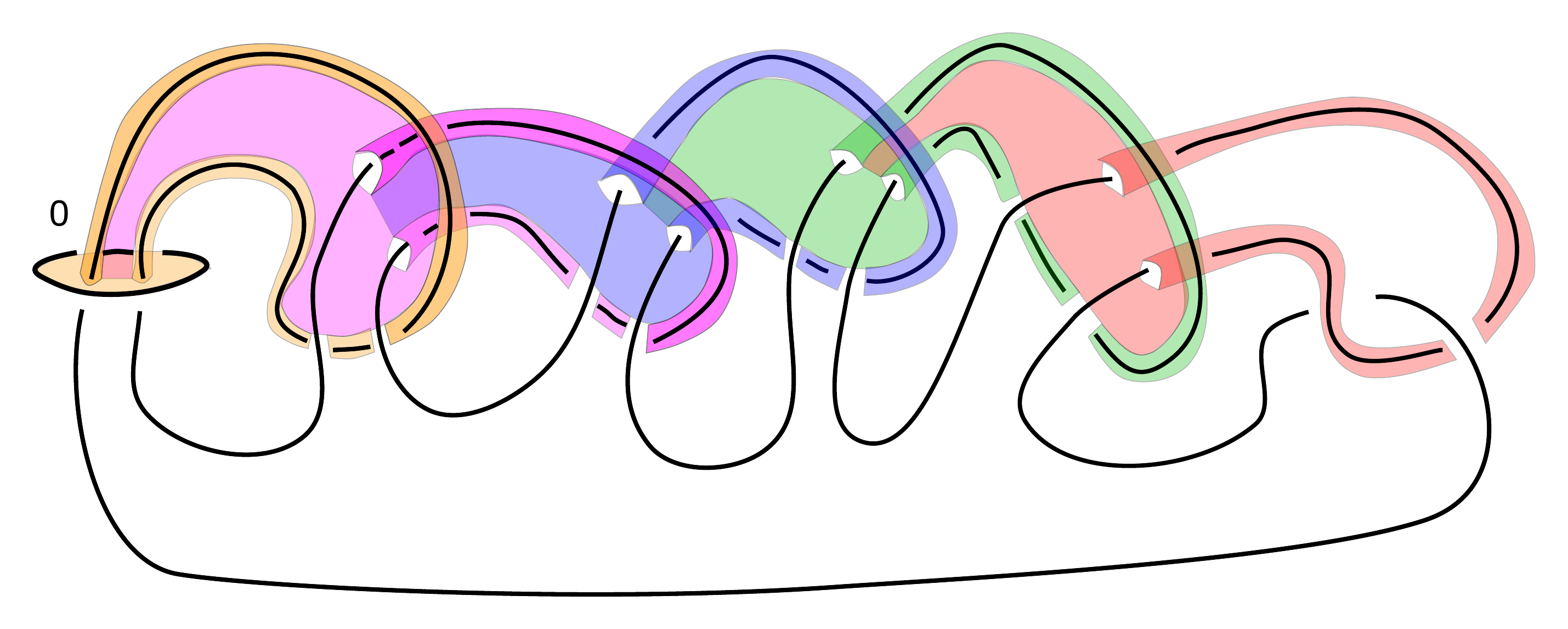}
\caption{An embedded order $6$ half-surface tower  in $ S^1 \times S^2 \setminus \nu( J)$ whose bottom stage generates $H_2(S^1 \times S^2 \setminus \nu( J))$.}
\label{fig:class6}
\end{figure}

Recall that our construction of $D(K)$ arose originally by studying the effect of the knotification construction on a link. 

\begin{definition}[Knotified Link \cite{ozszholoknot} ] \label{def:knotifiedlink}
Let L be an n-component link in $S^3$. The knotified version of this link is denoted $\kappa(L)$, lies in $\#^{n-1} S^1 \times S^2$, and is obtained in the following way: 

Fix $n-1$ embedded copies of $S^0$ inside $L$ labeled $\{p_i, q_i\}$ so that if each $p_i$ and $q_i$ are identified, the resulting quotient of L is a connected graph. Now, view each pair as the feet of a 4-dimensional 1-handle we can attach to $B^4$, and note that looking at the boundary of the result we get L inside $\#^{n-1} S^1 \times S^2$. We can now band sum the components of L together inside the boundary of these 1-handles to get a knot $\kappa(L)$ inside  $\#^{n-1} S^1 \times S^2$. We call $\kappa(L)$ the knotification of the link L.
\end{definition}

\begin{figure}[ht]\centering
\begin{subfigure}{.45\textwidth}
\centering
\includegraphics[scale=.32]{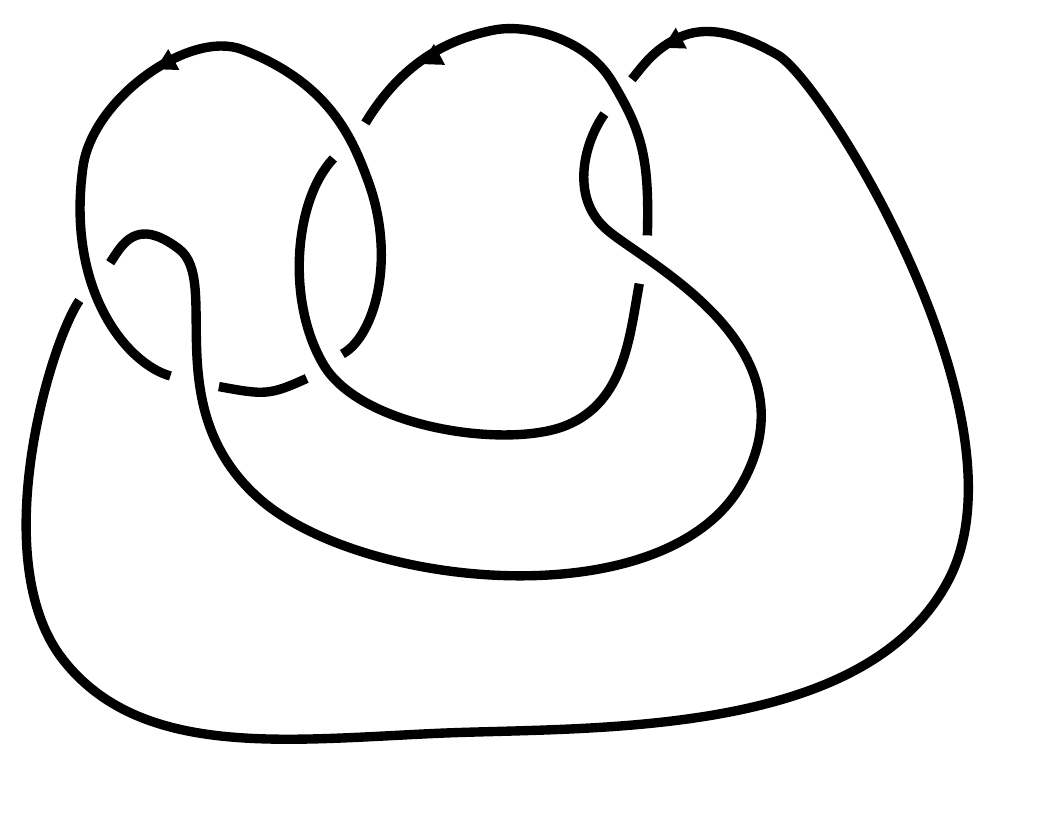}

\end{subfigure}
\begin{subfigure}{.45\textwidth}
\centering
\includegraphics[scale=.29]{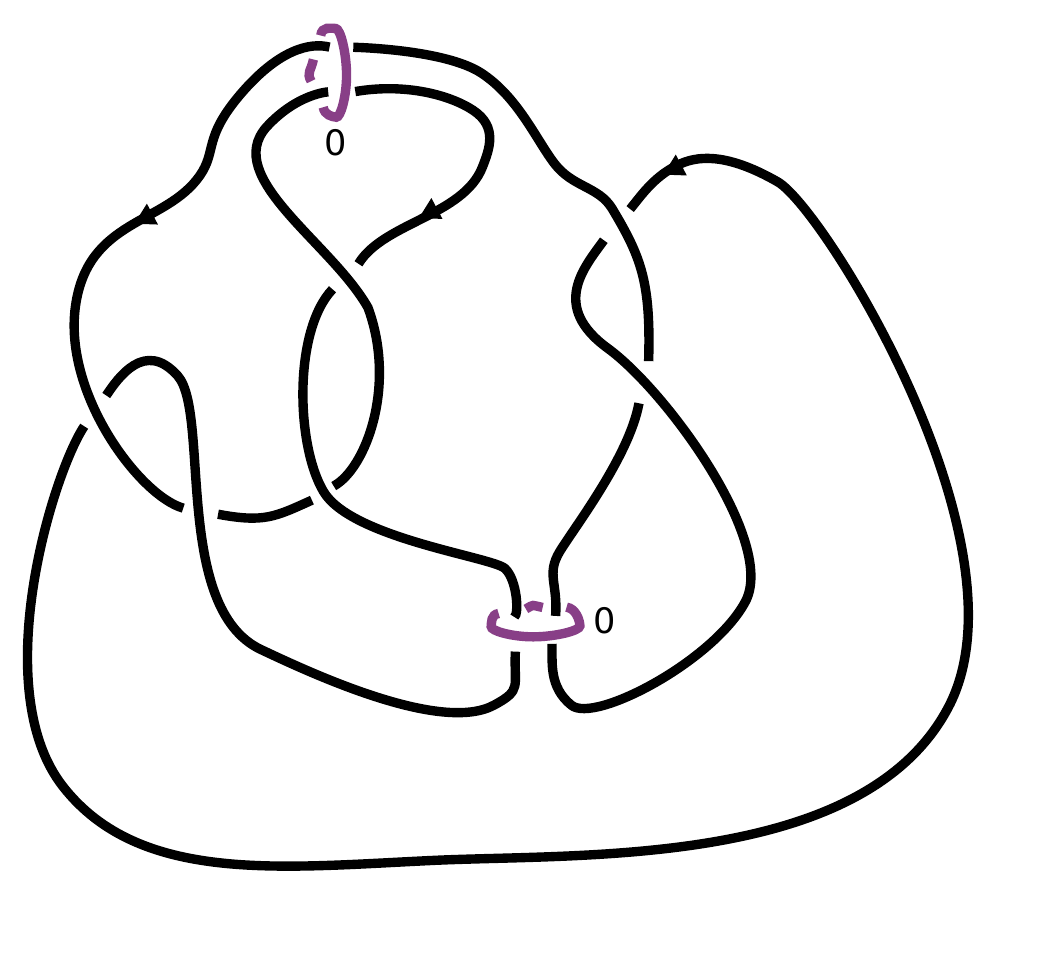}

\end{subfigure}
 \caption{A link $L \subset S^3$ and its knotification $\kappa(L) \subset S^1 \times S^2 \# S^1 \times S^2$.}
\end{figure}

To address this case, we must first introduce the following notion.

\begin{definition}[Interior band sum \cite{cochranmemoir}] Given a link $L$ of $m$ components and disjoint, oriented, embedded bands $b_1, ... , b_k$ in $S^3$ homeomorphic to $[0,1] \times [-1,1]$ whose intersections with $L$ are along the initial and terminal arcs $b_i(\{0,1\} \times [-1,1])$ which lie in $L$ and have the opposite orientation, a new (polygonal) link can be defined by deleting the collection of arcs $\{0,1\} \times [-1,1]$ and replacing them with $[0,1] \times \{-1,1\}$.
\end{definition}

This definition allows us to use a slight restatement of Theorem 8.9 from \cite{cochranmemoir} which we will use in calculations. The restatement also corrects an indexing error in the original proof. 

\begin{thm}[Cochran \cite{cochranmemoir}]  Suppose that $b(L)$ is an interior band sum involving $k$ bands. If the first non-vanishing $\bar{\mu}$-invariant of L is weight $\leq r < \infty$, then the first non-vanishing $\bar{\mu}$-invariant of $b(L)$ is weight greater than $\left \lfloor{\frac{r}{(k+1)}}\right \rfloor$.
\end{thm}

We can now prove the following bound on the Dwyer number of a knotified link.

\knotifiedbound
\begin{proof}
The knotification of $L$ can also be constructed using the following process. First, take the disjoint union of $L$ with an $n-1$-component unlink $U$ and perform an interior band sum on the sublink $L \subset L \sqcup U$ 
with $n-1$ bands, one through each component of the added unlink. We will call the resulting link $J$. Notice that $U$ is unchanged by this procedure and $L$ has fused to become a one-component sublink which we call $L^{\prime}$.  We finally arrive at the knotification by performing $0$-surgery on the sublink of $L^{\prime}$ corresponding to $U$. By the above theorem of Cochran, the first nonzero $\bar{\mu}_{J}$ is weight $r+1$. The result then follows from Theorem \ref{thm:computationthm}.
\end{proof}

We conclude this section by illustrating the following properties of $D(K)$.

\begin{prop} In $\inlineMl$ for any $l\in \Z_{+}$, every null-homologous knot $K \subset \inlineMl$ has $D(K) \geq 3$. Moreover, every integer larger than $3$ can be realized as the Dwyer number of a slice knot in $\inlineMl$. 
 \end{prop}
\begin{proof} 
Since $K$ is null-homologous, from Lemma \ref{gppresent} we can see that the only relation in $G/G_q$ not coming from a product of simple commutators in $F_q$ is $[x, R_q(l)]$ where $R_q(l)$ is freely homotopic to a $0$-framed longitude of $K$ modulo $F_q$ and is therefore in $F_2$. Thus, $[x, R_q(l)] \in F_3$ and $D(K) \geq 3$.
If $K$ is the unknot, $\inlineMl \setminus \nu(K)$ is homeomorphic to the connected sum of $\inlineMl$ with a solid torus. Thus every generator of $\text{H}_2(\inlineMl \setminus \nu(K))$ can be represented by a map of a half-surface tower of order $m$ of arbitrary class. 

To see that every integer $i$ larger than $3$ is realized as the Dwyer number of a knot,we will construct a simple family of knots $J_i$. First, consider the oriented Hopf link inside $S^3$. To get $J_3$, double one component of it to get the Borromean rings and perform $0$-surgery on two components as shown in Figure \ref{fig:familyJ}. Since the Borromean rings have first nonzero $\bar{\mu}$-invariant $\bar{\mu}(123)$, by Theorem \ref{thm:computationthm} $D(J_3) = 3$. Iterate this Bing-doubling process on the doubled Hopf link as shown in figure \ref{fig:familyJ} and perform $0$-surgery on the sublink leaving out exactly one new component after this doubling procedure. 

Notice that the underlying link in the surgery diagram is the result of the ``Bing doubling along a tree" procedure described in \cite{cochranmemoir} and thus the resulting link from doubling a single component of $J_i$ to get $J_{i+1}$ has a non-vanishing Milnor invariant of weight $i+1$. Additionally, all Milnor invariants of smaller weight vanish, and therefore by Theorem \ref{thm:computationthm} $D(J_{i})=i$ for all $i \geq 3$. Each of these knots $J_i$ bounds an obvious disk in $\natural^{i-1} D^2 \times S^2$ as in Figure \ref{fig:disk}.
\end{proof}

\begin{figure}[h] 
\centering
\begin{subfigure}[b]{0.45\textwidth}
\centering
  \includegraphics[width=.6\linewidth]{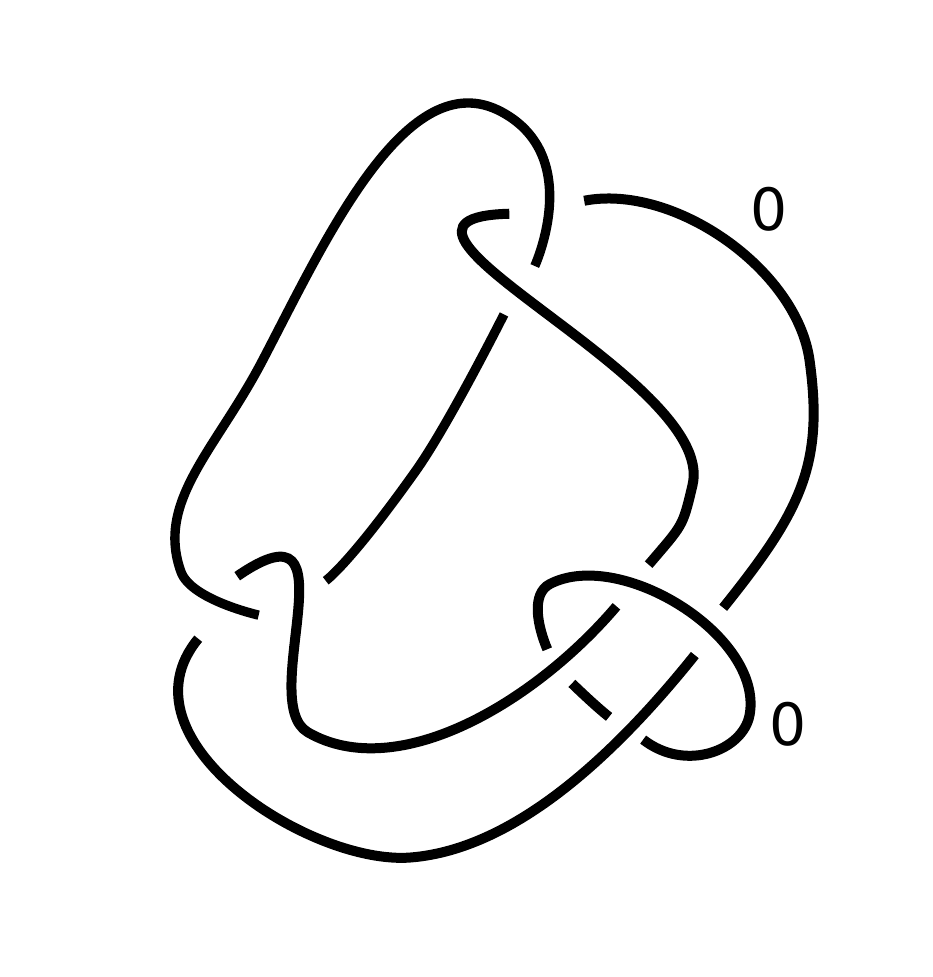}

\end{subfigure}
\begin{subfigure}[b]{0.45\textwidth}
\centering
  \includegraphics[width=.6\linewidth]{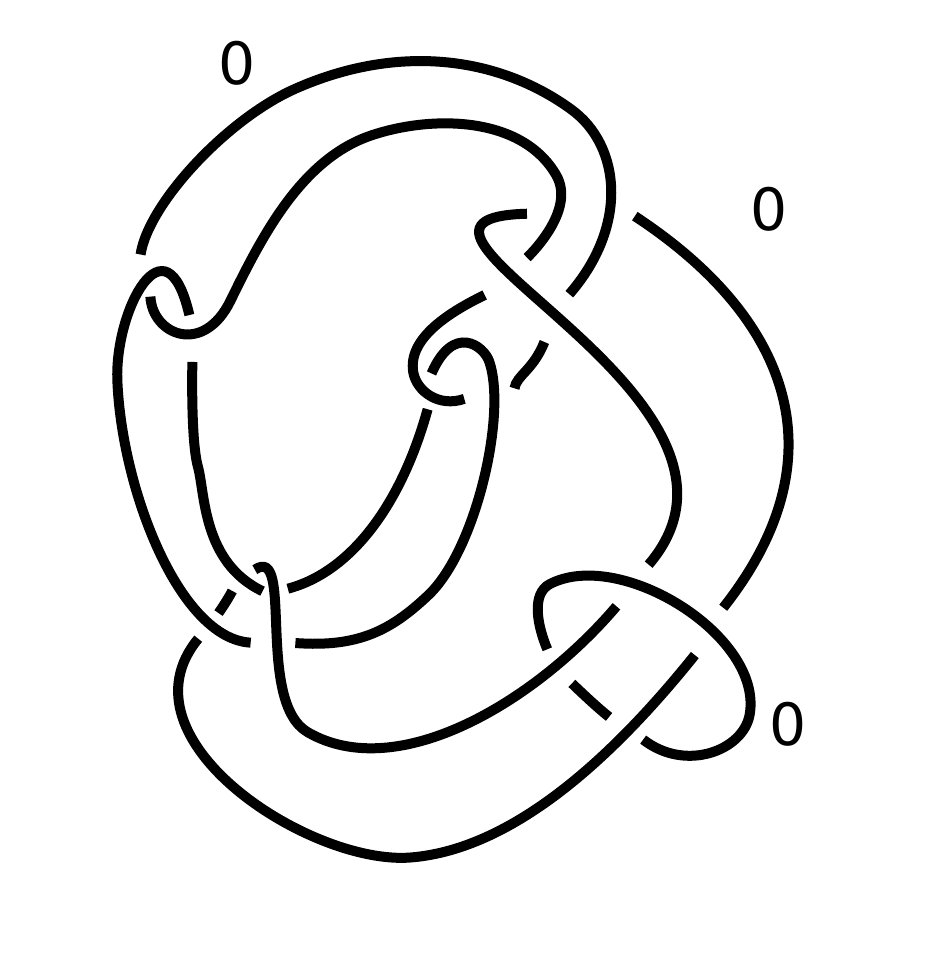}

\end{subfigure}
\caption{$J_3 \subset \#^2 S^1 \times S^2$ and $J_4 \subset \#^3 S^1 \times S^2$.}
\label{fig:familyJ}
\end{figure}

\begin{figure}[h] 
\centering
  \includegraphics[width=.25\linewidth]{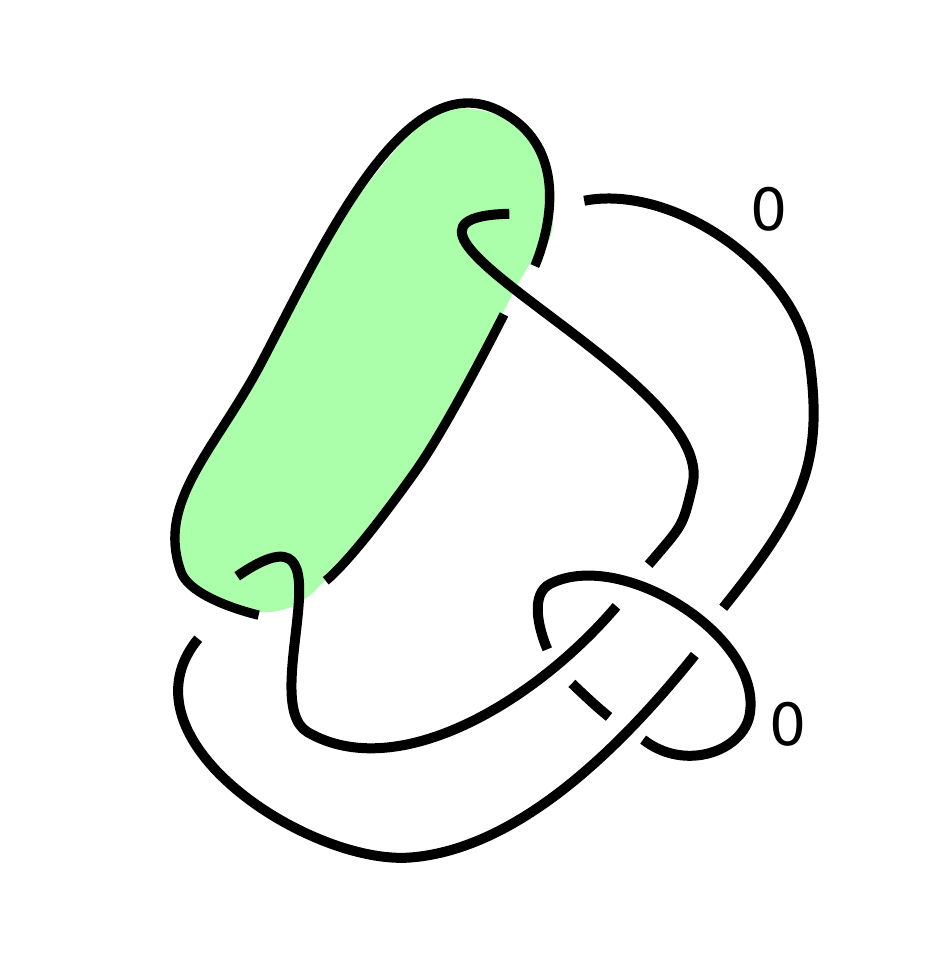}
\caption{A slice disk for $K_2$.}
\label{fig:disk}
\end{figure}

We also see that our invariant has applications to recent work of Celoria in \cite{celoriaac}.

\begin{definition} 
A knot $K$ inside a $3$-manifold $Y$ is local if there is an embedded $B^3 \subset Y$ such that $K \subset B^3$. Similarly, a link $L$ inside a $3$-manifold $Y$ is local if there is an embedded $B^3 \subset Y$ such that $L \subset B^3$
\end{definition}

\begin{definition}[Celoria \cite{celoriaac}]
Two knots $K_0$ and $K_1$ in a $3$-manifold $Y$ are almost concordant $K_0 \sim_{ac} K_1$ if there are local knots $K_0^{\prime}$ and $K_1^{\prime}$ in $Y$ such that $K_0 \# K_0^{\prime}$ is concordant to $K_1 \# K_1^{\prime}$ in $Y \times I$.
\end{definition}

Now, we can see the following.

\almostconc
\begin{proof}
There is a surgery diagram for $K$ as an $l+1$-component link in $S^3$ with $0$-surgery performed on $l$ of the components (note this diagram is certainly not unique). By Theorem \ref{thm:computationthm}, $D(K_i)$ is exactly the weight of the first non-vanishing Milnor invariant of this link in $S^3$ which we will call $L$. By abusing notation, let $K^{\prime}$ also refer to the image of the local knot inside $S^3$. We then see that infecting the component of $L$ corresponding to $K$ after surgery by the knot $K_i^{\prime}$ results in a link $L_i^{\prime}$ with all of the same Milnor invariants as $L_i$ by Otto in Proposition 4.2 of \cite{cottothesis}. Finally, it is clear that there is a $0$-surgery on an $l$-component sublink of $L_i^{\prime}$ gives the sum $K \# K^{\prime} \subset \inlineMl$ and thus applying Theorem \ref{thm:computationthm} again we have $D(K) = D(K \# K^{\prime})$ when $K^{\prime}$ is a local knot.
\end{proof}

Therefore Theorem \ref{thm:realization} also gives us families of knots in $\#^i S^1 \times S^2$ for each $i$ which not almost-concordant in $\#^i S^1 \times S^2$ to the unknot (or each other).

%
%
%
\bibliographystyle{gtart}

\bibliography{AGTbib} 

\end{document}